\documentclass[a4paper,
fontsize=11pt,%
oneside,%
numbers=enddot]{scrartcl}
\KOMAoptions{DIV=12} % size of type area (see user guide)
\usepackage[utf8]{inputenc}
\usepackage[T1]{fontenc}
\usepackage{textcomp}
\usepackage[fleqn]{amsmath}
\usepackage{amsthm}
\usepackage{amssymb}
\usepackage{amsfonts}
\usepackage{graphicx}
\usepackage{libertine}
\usepackage[libertine,
slantedGreek,
nosymbolsc,
nonewtxmathopt,
subscriptcorrection]{newtxmath}
\usepackage[scaled=0.95,varqu,varl]{inconsolata}
\usepackage[scr=boondox]{mathalpha} % for \mathscr
\usepackage{euscript} % for \EuScript
\usepackage{leftindex}
\allowdisplaybreaks[1] % allows align to split
\numberwithin{equation}{section} % must be before cleveref!
\usepackage{xcolor}
\definecolor{dblue}{HTML}{0455BF}
\definecolor{dgreen}{HTML}{02724A}
\definecolor{dgreen2}{HTML}{025951}
\definecolor{dred}{HTML}{D90404}
\definecolor{dviolet}{HTML}{42208C}
\definecolor{labelkey}{HTML}{025951}
\definecolor{refkey}{HTML}{025951}
\definecolor{orng}{HTML}{D35400}
\definecolor{pblue}{rgb}{0.1176,0.5647,1}
\definecolor{pgreen}{rgb}{0.1961,0.8039,0.1961}
\definecolor{pred}{rgb}{1.0,0.2706,0.0}
\definecolor{fred}{rgb}{0.98,0.40,0.93}
\definecolor{pyellow}{rgb}{1.0,0.6471,0.0}
\usepackage{tikz}
\usepackage{pgfplots}
\usepgfplotslibrary{colormaps}
\pgfplotsset{colormap={setting1}{color=(pblue) color=(pyellow) 
color=(pred)},
colormap={setting2}{rgb255=(255,0,0) rgb255=(255,255,0)}}
\addtokomafont{section}{\centering}

\usepackage{enumitem}
\setlist{itemsep=-2.0pt}

\usepackage{upref} % upright number when using ref in theorems
\usepackage{hyperref}
\hypersetup{colorlinks=true,
linktocpage=true,
linkcolor=dblue,
citecolor=dgreen,
urlcolor=dred,
pdfencoding=auto,
hypertexnames=false}
\makeatletter
\g@addto@macro\th@plain{
\thm@headfont{\bfseries\sffamily}
\thm@notefont{}}
\g@addto@macro\th@definition{
\thm@headfont{\bfseries\sffamily}
\thm@notefont{}}
\g@addto@macro\th@remark{
\thm@headfont{\bfseries\sffamily}
\thm@notefont{}}
\makeatother
\theoremstyle{plain}
\newtheorem{theorem}{Theorem}[section]
\newtheorem{proposition}[theorem]{Proposition}
\newtheorem{corollary}[theorem]{Corollary}
\newtheorem{lemma}[theorem]{Lemma}

\theoremstyle{definition}
\newtheorem{definition}[theorem]{Definition}
\newtheorem{example}[theorem]{Example}

\theoremstyle{remark}
\newtheorem{remark}[theorem]{Remark}

\usepackage[extdef=true]{delimset}
\DeclareMathDelimiterSet{\scal}[2]{
\selectdelim[l]<{#1}
\mathpunct{}\selectdelim[p]|
{#2}\selectdelim[r]>}

\newcommand{\menge}[2]{\bigl\{{#1}\mid{#2}\bigr\}} 
\DeclareMathDelimiterSet{\Menge}[2]{\selectdelim[l]\{
{#1}\selectdelim[m]|{#2}\selectdelim[r]\}}
\newcommand{\RR}{\mathbb{R}}
\newcommand{\NN}{\mathbb{N}}
\newcommand{\HH}{\mathcal{H}}
\newcommand{\GG}{\mathcal{G}}
\newcommand{\KK}{\mathcal{K}}
\newcommand{\XX}{\mathcal{X}}
\newcommand{\BL}{\ensuremath{\EuScript B}\,}
\newcommand{\pinf}{{+}\infty}
\newcommand{\minf}{{-}\infty}
\newcommand{\zeroun}{\intv[o]{0}{1}}

\newcommand{\RXX}{\intv{\minf}{\pinf}}
\newcommand{\RX}{\intv[l]0{\minf}{\pinf}}
\newcommand{\RP}{\intv[r]0{0}{\pinf}}
\newcommand{\RPP}{\intv[o]0{0}{\pinf}}
\newcommand{\RPX}{\intv0{0}{\pinf}}

\newcommand{\emp}{\varnothing}
\newcommand{\pushfwd}{\ensuremath{\mathbin{\raisebox{-0.20ex}
{\mbox{\LARGE$\triangleright$}}}}}

\newcommand{\Id}{\mathrm{Id}}

\newcommand{\moyo}[2]{\leftindex[I]^{#2}{#1}}

\DeclareMathOperator{\ran}{ran}
\DeclareMathOperator{\cdom}{\overline{dom}}
\DeclareMathOperator{\dom}{dom}

\DeclareMathOperator{\gra}{gra}
\DeclareMathOperator{\zer}{zer}

\DeclareMathOperator{\sri}{sri}

\DeclareMathOperator{\proj}{proj}
\DeclareMathOperator{\haus}{haus}

\newcommand{\lev}[1]%
{{\ensuremath{{{{\operatorname{lev}}}_{\leq #1}}\,}}}

\DeclareFontFamily{U}{mathb}{}
\DeclareFontShape{U}{mathb}{m}{n}{<-5.5> mathb5 <5.5-6.5> mathb6 
<6.5-7.5> mathb7 <7.5-8.5> mathb8 <8.5-9.5> mathb9 <9.5-11> mathb10
<11-> mathb12}{}
\DeclareSymbolFont{mathb}{U}{mathb}{m}{n}
\DeclareFontSubstitution{U}{mathb}{m}{n}
\DeclareMathSymbol{\blackdiamond}{\mathbin}{mathb}{"0C}
\DeclareMathOperator{\rav}{{\mathsf{rav}}}
\renewcommand{\leq}{\leqslant}
\renewcommand{\geq}{\geqslant}
\newcommand{\exi}{\exists\,}

\newcommand{\proxc}[1]{\mathbin%
{\ensuremath{\overset{#1}{\diamond}}}}
\newcommand{\proxcc}[1]{\mathbin%
{\ensuremath{\overset{#1}{\blackdiamond}}}}
\newcommand{\Rm}[1]{\overset{\mathord{\diamond}}{\mathsf{M}}_{#1}}
\newcommand{\Rcm}[1]{\overset{\mathord{\blackdiamond}}{
\mathsf{M}}_{#1}}
\renewenvironment{abstract}{%
\vspace*{-0.50cm}
\small
\quotation%
\noindent%
{\normalfont\bfseries\sffamily
\nobreak\abstractname\ }%
}{%
\endquotation%
\medskip
}
\renewcommand{\abstractname}{Abstract.}

\newcommand\mscsname{MSC classification.}
\newenvironment{keywords}
{\renewcommand\abstractname{\keywordsname}\begin{abstract}}
{\end{abstract}}
\newenvironment{MSC}
{\renewcommand\abstractname{\mscsname}\begin{abstract}}
{\end{abstract}}
\usepackage[auth-sc]{authblk}
\newcommand{\email}[1]{\href{mailto:#1}{\nolinkurl{#1}}}
\renewcommand*\Affilfont{\normalfont\normalsize}
\newcommand\affilcr{\protect\\ \protect\Affilfont}
\makeatletter
\renewcommand\AB@affilsepx{\protect\\[0.5em]}
\makeatother
\author[]{Diego J. Cornejo}
\affil[]{North Carolina State University
\affilcr
Department of Mathematics
\affilcr
Raleigh, NC 27695, USA
\affilcr
\email{djcornej@ncsu.edu}
}

\begin{document}

\title{%
Parametrized Families of Resolvent Compositions
\thanks{%
This work was supported by the National
Science Foundation under grant CCF-2211123.
}}

\date{~}
\maketitle

\vspace{12mm}

\begin{abstract} 
This paper presents an in-depth analysis of a parametrized
version of the resolvent composition, an operation that combines a
set-valued operator and a linear operator. We provide new
properties and examples, and show that resolvent compositions can
be interpreted as parallel compositions of perturbed operators.
Additionally, we establish new monotonicity results, even in cases
when the initial operator is not monotone. Finally, we derive
asymptotic results regarding operator convergence, specifically
focusing on graph-convergence and the $\rho$-Hausdorff distance.
\end{abstract}

\begin{keywords}
Graph-convergence,
monotone operator,
parallel composition,
resolvent composition,
resolvent mixture,
$\rho$-Hausdorff distance, 
set-convergence.
\end{keywords}

\begin{MSC}
47H04,
47H05,
49J53
\end{MSC}

\newpage

\section{Introduction}
\label{sec:1}

Throughout, $\HH$ is a real Hilbert space with power set $2^\HH$,
identity operator $\Id_{\HH}$, scalar product
$\scal{\cdot}{\cdot}_{\HH}$, and associated norm $\|\cdot\|_{\HH}$.
In addition, $\GG$ is a real Hilbert space, the space of bounded
linear operators from $\HH$ to $\GG$ is denoted by $\BL(\HH,\GG)$,
and $\BL(\HH)=B(\HH,\HH)$. Let $L\in\BL(\HH,\GG)$. The adjoint of
$L$ is denoted by $L^*$, and the parallel composition of a
set-valued operator $B\colon\GG\to2^\GG$ by $L^*$ is the operator
from $\HH$ to $2^\HH$ given by
\begin{equation}
\label{e:parallel}
L^*\pushfwd B=\brk1{L^*\circ B^{-1}\circ L}^{-1}.
\end{equation}

We focus our attention on new methods to combine a set-valued
operator with a linear operator, which have recently been
introduced in \cite{Svva23}, where they have been studied only for
the case $\gamma=1$. 

\begin{definition}
\label{d:1}
Let $L\in\BL(\HH,\GG)$, let $B\colon\GG\to2^\GG$, and let
$\gamma\in\RPP$. The \emph{resolvent composition} of $B$ and $L$
with parameter $\gamma$ is the operator
$L\proxc{\gamma}B\colon\HH\to2^\HH$ given by 
\begin{equation}
\label{e:d1}
L\proxc{\gamma}B
=L^*\pushfwd(B+\gamma^{-1}\Id_{\GG})-\gamma^{-1}\Id_{\HH}
\end{equation}
and the \emph{resolvent cocomposition} of $B$ and $L$ with
parameter $\gamma$ is the operator
$L\proxcc{\gamma}B\colon\HH\to2^\HH$ given by 
$L\proxcc{\gamma}B=(L\proxc{1/\gamma}B^{-1})^{-1}$. Further,
$L\proxc{}B=L\proxc{1}B$ and $L\proxcc{}B=L\proxcc{1}B$.
\end{definition}

Resolvents of set-valued operators are essential in the numerical
solution of monotone inclusion problems
\cite{Acnu24,Tsp21,Ecks92,Glow89,Spin83,Tsen90}. A motivation for
studying the resolvent compositions of Definition~\ref{d:1} stems
from the fact that their resolvent can be computed explicitly,
unlike those of the standard composite operators $L^*\circ B\circ
L$ and $L^*\pushfwd B$, for which the resolvent is typically
intractable and requires dedicated numerical methods
\cite{Arta21,Joca09,Jmaa11}. Resolvent compositions also show up in
relaxations of inconsistent inclusion problems \cite{Jota24,Svva23}.
For instance, these new composite operators can be utilized to
model relaxations of convex feasibility and nonlinear
reconstruction problems \cite{Siim22}. Furthermore, the resolvent
composition of the subdifferential of a proper lower semicontinuous
convex function is the subdifferential of a function called the
{\em proximal composition} (see \cite{Jota24,Svva23,Eect24}), which
has been used in image recovery and machine learning applications
\cite{Eusi24}.

The goal of this paper is to present an in-depth analysis of the
parametrized compositions of Definition~\ref{d:1}. We provide
various new properties and examples, as well as connections with 
connection with the parallel composition $L^*\pushfwd B$ and the
standard composition $L^*\circ B\circ L$. Additionally, we
establish new monotonicity results, including the preservation of
monotonicity, strongly monotonicity, and maximally monotonicity,
and examine the Fitzpatrick function of the resolvent composition.
Finally, we investigate the convergence of the operators
$L\proxc{\gamma}B$ and $L\proxcc{\gamma}B$ as $\gamma$ varies, by
examining the graph-convergence and the $\rho$-Hausdorff distance
convergence.
 
The remainder of the paper is organized as follows. In
Section~\ref{sec:2}, we provide our notation and necessary
mathematical background. In Section~\ref{sec:3}, we investigate
new properties of the parametrized resolvent compositions and
present several examples. Section~\ref{sec:4} is devoted to study
the monotonicity of resolvent compositions. Finally,
Section~\ref{sec:5} provides convergence results for parametrized
resolvent compositions as the parameter varies.

\section{Notation and background}
\label{sec:2}

We first present our notation, which follows \cite{Livre1}.

Let $A\colon\HH\to2^\HH$ be a set-valued operator. The domain of
$A$ is $\dom A=\menge{x\in\HH}{Ax\neq\emp}$, the range of $A$ is
$\ran A=\menge{x^*\in\HH}{(\exi x\in\HH)\,x^*\in Ax}$, the graph of
$A$ is $\gra A=\menge{(x,x^*)\in\HH\times\HH}{x^*\in Ax}$, the
zeros of $A$ is $\zer A=\menge{x\in\HH}{0\in Ax}$, the inverse of
$A$ is the set-valued operator $A^{-1}$ with graph
$\menge{(x^*,x)\in\HH\times\HH}{x^*\in Ax}$.
The resolvent of $A$ is 
\begin{equation}
\label{e:res}
J_{A}=(\Id_{\HH}+A)^{-1}
\end{equation}
and the Yosida approximation of $A$ of index $\gamma\in\RPP$ is
\begin{equation}
\label{e:moyo}
\moyo{A}{\gamma}=\gamma^{-1}(\Id_{\HH}-J_{\gamma A})
=(A^{-1}+\gamma\Id_{\HH})^{-1}.
\end{equation}
In particular, when $\gamma=1$,
\begin{equation}
\label{e:23}
\Id_{\HH}-J_A=J_{A^{-1}}.
\end{equation}
The operator $A$ is monotone if
\begin{equation}
\brk1{\forall(x_1,x_1^*)\in\gra A}
\brk1{\forall(x_2,x_2^*)\in\gra A}\,\,
\scal{x_1-x_2}{x_1^*-x_2^*}_{\HH}\geq 0,
\end{equation}
$\alpha$-strongly monotone for some $\alpha\in\RPP$ if
$A-\alpha\Id_{\HH}$ is monotone, and maximally monotone if it is
monotone and there exists no monotone operator
$B\colon\HH\to2^\HH$ such that $\gra B$ properly
contains $\gra A$. 

Let $L\in\BL(\HH,\GG)$. If $\ran L$ is closed, the generalized (or
Moore--Penrose) inverse of $L$ is denoted by $L^\dagger$. Further,
$L$ is an isometry if $L^*\circ L=\Id_{\HH}$ and a coisometry if 
$L\circ L^*=\Id_{\GG}$. 

Let $D$ be a nonempty subset of $\HH$ and let $T\colon D\to\HH$.
Then $T$ is nonexpansive if
\begin{equation}
(\forall x\in D)(\forall y\in D)\quad
\|Tx-Ty\|_{\HH}\leq\|x-y\|_{\HH},
\end{equation}
and firmly nonexpansive if $2T-\Id_{\HH}$ is nonexpansive.

The Legendre conjugate of $f\colon\HH\to\RXX$ is the function
\begin{equation}
f^*\colon\HH\to\RXX\colon x^*\mapsto
\sup_{x\in\HH}\brk1{\scal{x}{x^*}_{\HH}-f(x)}
\end{equation}
and the Moreau envelope of $f\colon\HH\to\RX$ of parameter
$\gamma\in\RPP$ is
\begin{equation}
\moyo{f}{\gamma}\colon\HH\to\RXX\colon x\mapsto
\inf_{y\in\HH}\brk2{f(y)+\dfrac{1}{2\gamma}\|x-y\|_{\HH}^2}.
\end{equation}
A function $f\colon\HH\to\RX$ is proper if 
$\dom f=\menge{x\in\HH}{f(x)<\pinf}\neq\emp$. 
The set of proper lower semicontinuous convex functions from $\HH$
to $\RX$ is denoted by $\Gamma_0(\HH)$. The subdifferential of a
function  $f\in\Gamma_0(\HH)$ is
\begin{equation}
\partial f\colon\HH\to2^\HH\colon x\mapsto
\menge{x^*\in\HH}{
(\forall z\in\HH)\,\scal{z-x}{x^*}_{\HH}+f(x)\leq f(z)},
\end{equation}
and its inverse is
\begin{equation}
\label{e:29}
(\partial f)^{-1}=\partial f^*.
\end{equation}

Let $C$ be a nonempty convex subset of $\HH$. The normal cone of
$C$ is denoted by $N_C$ and the strong relative interior of $C$ is
denoted by $\sri C$. Additionally, if $C$ is closed, the projection
operator onto $C$ is denoted by $\proj_C$. Finally, the closed ball
with center $x\in\HH$ and radius $\rho\in\RPP$ is denoted by
$B(x;\rho)$.
 
The following facts will be used in the paper.

\begin{lemma}
\label{l:3}
Let $L\in\BL(\HH,\GG)$, let $A\colon\HH\to2^\HH$, let
$U\colon\GG\to\GG$, and let $\rho\in\RPP$. Then the following hold:
\begin{enumerate}
\item
\label{l:3i}
$\dom(L\pushfwd A)\subset L(\dom A)$.
\item
\label{l:3ii}
Let $\KK$ be a real Hilbert space and let $S\in\BL(\GG,\KK)$. Then
$S\pushfwd(L\pushfwd A)=(S\circ L)\pushfwd A$.
\item
\label{l:3iii}
$\rho(L\pushfwd A)=L\pushfwd(\rho A)$.
\item
\label{l:3iv}
$(L\pushfwd A)(\rho\Id_{\GG})=L\pushfwd(A(\rho\Id_{\HH}))$.
\item
\label{l:3v}
$L\pushfwd A+U=L\pushfwd(A+L^*\circ U\circ L)$.
\end{enumerate}
\end{lemma}
\begin{proof}
\ref{l:3i}: By \eqref{e:parallel},
$\dom(L\pushfwd A)=\ran(L\circ A^{-1}\circ L^*)\subset
L(\ran A^{-1})=L(\dom A)$.

\ref{l:3ii}: \cite[Proposition~25.42(ii)]{Livre1}

\ref{l:3iii}: Since $A^{-1}(\rho^{-1}\Id_{\HH})=(\rho A)^{-1}$, it
follows from \eqref{e:parallel} that $\rho(L\pushfwd A)
=(L\circ A^{-1}\circ L^*(\rho^{-1}\Id_{\GG}))^{-1}
=(L\circ(A^{-1}(\rho^{-1}\Id_{\HH}))\circ L^*)^{-1}
=(L\circ(\rho A)^{-1}\circ L^*)^{-1}
=L\pushfwd(\rho A)$.

\ref{l:3iii}: Since $\rho^{-1}A^{-1}=(A(\rho\Id_{\HH}))^{-1}$, it
follows from \eqref{e:parallel} that
$(L\pushfwd A)(\rho\Id_{\GG})
=(\rho^{-1}L\circ A^{-1}\circ L^*)^{-1}
=(L\circ(\rho^{-1}A^{-1})\circ L^*)^{-1}
=(L\circ(A(\rho\Id_{\HH}))^{-1}\circ L^*)^{-1}
=L\pushfwd(A(\rho\Id_{\HH}))$.

\ref{l:3v}: Let $x\in\HH$ and set $M=L\pushfwd A+U$.
It follows from \eqref{e:parallel} that
\begin{align}
x^*\in Mx&\Leftrightarrow 
x^*-Ux\in\brk1{L\circ A^{-1}\circ L^*}^{-1}x\nonumber\\
&\Leftrightarrow x\in L\brk2{A^{-1}\brk1{L^*x^*-L^*(Ux)}}
\nonumber\\
&\Leftrightarrow(\exi y\in\GG)\quad 
y\in A^{-1}\brk1{L^*x^*-L^*(Ux)}\quad\text{and}\quad x=Ly
\nonumber\\
&\Leftrightarrow(\exi y\in\GG)\quad 
L^*x^*\in Ay+L^*\brk1{U(Ly)}\quad\text{and}\quad x=Ly
\nonumber\\
&\Leftrightarrow(\exi y\in\GG)\quad y\in
\brk1{A+L^*\circ U\circ L}^{-1}(L^*x^*)\quad\text{and}\quad x=Ly
\nonumber\\
&\Leftrightarrow 
x\in L\brk2{\brk1{A+L^*\circ U\circ L}^{-1}(L^*x^*)}
\nonumber\\
&\Leftrightarrow x^*\in\brk1{L\pushfwd(A+L^*\circ U\circ L)}x,
\end{align}
which completes the proof.
\end{proof}

\section{Resolvent compositions}
\label{sec:3}

This section outlines the general properties of the parametrized
compositions of Definition~\ref{d:1}, which will be utilized
subsequently.

\begin{proposition}
\label{p:1}
Let $L\in\BL(\HH,\GG)$, let $B\colon\GG\to2^\GG$, let
$\gamma\in\RPP$, and let $\rho\in\RPP$. Then the following hold:
\begin{enumerate}
\item
\label{p:1i}
$\dom(L\proxc{\gamma}B)\subset L^*(\dom B)$.
\item
\label{p:1ii}
$\ran(L\proxcc{\gamma}B)\subset L^*(\ran B)$.
\item
\label{p:1iii}
$L\proxc{\gamma}B=(L\proxcc{1/\gamma}B^{-1})^{-1}$.
\item
\label{p:1iv}
$\rho(L\proxc{\gamma}B)=L\proxc{\gamma/\rho}(\rho B)$.
\item
\label{p:1v}
$(L\proxc{\gamma}B)(\rho\Id_{\HH})
=L\proxc{\gamma/\rho}(B(\rho\Id_{\GG}))$.
\item
\label{p:1vi}
$\rho(L\proxcc{\gamma}B)=L\proxcc{\gamma/\rho}(\rho B)$.
\item
\label{p:1vii}
$(L\proxcc{\gamma}B)(\rho\Id_{\HH})
=L\proxcc{\gamma/\rho}(B(\rho\Id_{\GG}))$.
\item
\label{p:1viii}
Let $z\in\HH$, let $w\in\GG$, and set 
$\tau_wB\colon x\mapsto B(x-w)$. Then the following hold:
\begin{enumerate}
\item
\label{p:1viiia}
$(L\proxc{\gamma}B)-z=L\proxc{\gamma}(B-Lz)$.
\item
\label{p:1viiib}
$\tau_z(L\proxcc{\gamma}B)=L\proxcc{\gamma}(\tau_{Lz}B)$.
\end{enumerate}
\item
\label{p:1ix}
Let $\KK$ be a real Hilbert space and $S\in\BL(\KK,\HH)$. Then the
following hold:
\begin{enumerate}
\item
\label{p:1ixa}
$S\proxc{\gamma}(L\proxc{\gamma}B)=(L\circ S)\proxc{\gamma}B$.
\item
\label{p:1ixb}
$S\proxcc{\gamma}(L\proxcc{\gamma}B)=(L\circ S)\proxcc{\gamma}B$.
\end{enumerate}
\item
\label{p:1x}
Set $\beta=\gamma/(1+\rho\gamma)$. Then
$L\proxc{\gamma}(B+\rho\Id_{\GG})=(L\proxc{\beta}B)+\rho\Id_{\HH}$.
\item
\label{p:1xi}
$\moyo{(L\proxcc{\gamma+\rho}B}{\rho})
=L\proxcc{\gamma}(\moyo{B}{\rho})$.
\item
\label{p:1xii}
$\moyo{(L\proxcc{\gamma}B)}{\gamma}
=L^*\circ(\moyo{B}{\gamma})\circ L$.
\item
\label{p:1xiii}
$\zer(L\proxcc{\gamma}B)=\zer(L^*\circ(\moyo{B}{\gamma})\circ L)$.
\end{enumerate}
\end{proposition}
\begin{proof}
\ref{p:1i}: 
By Definition~\ref{d:1} and Lemma~\ref{l:3}\ref{l:3i},
$\dom(L\proxc{\gamma}B)\subset L^*(\dom(B+\gamma\Id_{\GG}))
=L^*(\dom B)$.

\ref{p:1ii}: By Definition~\ref{d:1} and \ref{p:1i}, 
$\ran(L\proxcc{\gamma}B)=\dom(L\proxc{1/\gamma}B^{-1})\subset
L^*(\dom B^{-1})=L^*(\ran B)$.

\ref{p:1iii}: This follows from Definition~\ref{d:1}.

\ref{p:1iv}: It follows from Definition~\ref{d:1} and
Lemma~\ref{l:3}\ref{l:3iii} that
\begin{align}
\rho\brk1{L\proxc{\gamma}B}
&=\rho\brk1{L^*\pushfwd(B+\gamma^{-1}\Id_{\GG})}
-\rho\gamma^{-1}\Id_{\HH}\nonumber\\
&=L^*\pushfwd(\rho B+\rho\gamma^{-1}\Id_{\GG})
-\rho\gamma^{-1}\Id_{\HH}\nonumber\\
&=L\proxc{\gamma/\rho}(\rho B).
\end{align}

\ref{p:1v}: By Definition~\ref{d:1} and Lemma~\ref{l:3}\ref{l:3iv},
we obtain
\begin{align}
\brk1{L\proxc{\gamma}B}(\rho\Id_{\HH})
&=\brk1{L^*\pushfwd(B+\gamma^{-1}\Id_{\GG})}(\rho\Id_{\HH})
-\rho\gamma^{-1}\Id_{\HH}\nonumber\\
&=L^*\pushfwd\brk1{B(\rho\Id_{\GG})+\rho\gamma^{-1}\Id_{\GG}}
-\rho\gamma^{-1}\Id_{\HH}\nonumber\\
&=L\proxc{\gamma/\rho}\brk1{B(\rho\Id_{\GG})}.
\end{align}

\ref{p:1vi}: By Definition~\ref{d:1} and \ref{p:1v}, 
\begin{equation}
\rho\brk1{L\proxcc{\gamma}B}
=\rho\brk1{L\proxc{1/\gamma}B^{-1}}^{-1}
=\brk2{\brk1{L\proxc{1/\gamma}B^{-1}}(\Id_{\HH}/\rho)}^{-1}
=\brk1{L\proxc{\rho/\gamma}(\rho B)^{-1}}^{-1}
=L\proxcc{\gamma/\rho}(\rho B).
\end{equation}

\ref{p:1vii}: By Definition~\ref{d:1} and \ref{p:1iv},
\begin{equation}
\brk1{L\proxcc{\gamma}B}(\rho\Id_{\HH})
=\brk1{L\proxc{1/\gamma}B^{-1}}^{-1}(\rho\Id_{\HH})
=\brk2{\rho^{-1}\brk1{L\proxc{1/\gamma}B^{-1}}}^{-1}
=\brk2{L\proxc{\rho/\gamma}\brk1{B(\rho\Id_{\GG})}^{-1}}^{-1}
=L\proxcc{\gamma/\rho}\brk1{B(\rho\Id_{\GG})}.
\end{equation}

\ref{p:1viiia}: Set $U\colon\HH\to\HH\colon x\mapsto z$. By
Definition~\ref{d:1} and Lemma~\ref{l:3}\ref{l:3v},
\begin{align}
\brk1{L\proxc{\gamma}B}-z&=L^*\pushfwd(B+\gamma^{-1}\Id_{\GG}
-L\circ U\circ L^*)-\gamma^{-1}\Id_{\HH}\nonumber\\
&=L^*\pushfwd(B-Lz+\gamma^{-1}\Id_{\GG})-\gamma^{-1}\Id_{\HH}
\nonumber\\
&=L\proxc{\gamma}(B-Lz).
\end{align}

\ref{p:1viiib}: Since $\tau_wB=(B^{-1}+w)^{-1}$, we combine
Definition~\ref{d:1} and \ref{p:1viiia} to derive
\begin{equation}
\tau_z\brk1{L\proxcc{\gamma}B}
=\brk2{\brk1{L\proxc{1/\gamma}B^{-1}}+z}^{-1}
=\brk1{L\proxc{1/\gamma}(B^{-1}+Lz)}^{-1}
=\brk1{L\proxc{1/\gamma}(\tau_{Lz}B)^{-1}}^{-1}
=L\proxcc{\gamma}(\tau_{Lz}B).
\end{equation}

\ref{p:1ixa}: By Definition~\ref{d:1} and
Lemma~\ref{l:3}\ref{l:3ii},
\begin{align}
(L\circ S)\proxc{\gamma}B
&=(L\circ S)^*\pushfwd(B+\gamma^{-1}\Id_{\GG})-\gamma^{-1}\Id_{\KK}
\nonumber\\
&=S^*\pushfwd\brk1{L^*\pushfwd(B+\gamma^{-1}\Id_{\GG})}
-\gamma^{-1}\Id_{\KK}\nonumber\\
&=S^*\pushfwd\brk1{(L\proxc{\gamma}B)+\gamma^{-1}\Id_{\HH}}
-\gamma^{-1}\Id_{\KK}\nonumber\\
&=S\proxc{\gamma}\brk1{L\proxc{\gamma}B}.
\end{align}

\ref{p:1ixb}: We combine Definition~\ref{d:1} and \ref{p:1ixa} to
obtain
\begin{equation}
S\proxcc{\gamma}\brk1{L\proxcc{\gamma}B}
=\brk2{S\proxc{1/\gamma}\brk1{L\proxc{1/\gamma}B^{-1}}}^{-1}
=\brk2{(L\circ S)\proxc{1/\gamma}B^{-1}}^{-1}
=(L\circ S)\proxcc{\gamma}B.
\end{equation}

\ref{p:1x}: Since $\beta^{-1}=\gamma^{-1}+\rho$, we deduce from
Definition~\ref{d:1} that
\begin{align}
L\proxc{\gamma}(B+\rho\Id_{\GG})
&=L^*\pushfwd(B+\rho\Id_{\GG}+\gamma^{-1}\Id_{\GG})
-\gamma^{-1}\Id_{\HH}\nonumber\\
&=L^*\pushfwd(B+\beta^{-1}\Id_{\GG})-\beta^{-1}\Id_{\HH}
+\rho\Id_{\HH}\nonumber\\
&=\brk1{L\proxc{\beta}B}+\rho\Id_{\HH}.
\end{align}

\ref{p:1xi}: By \eqref{e:moyo}, Definition~\ref{d:1}, and
\ref{p:1x}, 
\begin{align}
\moyo{\brk1{L\proxcc{\gamma+\rho}B}}{\rho}&=
\brk2{\brk1{L\proxc{1/(\gamma+\rho)}B^{-1}}+\rho\Id_{\HH}}^{-1}
\nonumber\\
&=\brk1{L\proxc{1/\gamma}(B^{-1}+\rho\Id_{\GG})}^{-1}
\nonumber\\
&=L\proxcc{\gamma}(B^{-1}+\rho\Id_{\GG})^{-1}\nonumber\\
&=L\proxcc{\gamma}(\moyo{B}{\rho}).
\end{align}

\ref{p:1xii}: It follows from \eqref{e:moyo} and 
Definition~\ref{d:1} that
\begin{equation}
\moyo{\brk1{L\proxcc{\gamma}B}}{\gamma}
=\brk2{\brk1{L\proxc{1/\gamma}B^{-1}}+\gamma\Id_{\HH}}^{-1}
=\brk1{L^*\pushfwd(B^{-1}+\gamma\Id_{\GG})}^{-1}
=L^*\circ(\moyo{B}{\gamma})\circ L.
\end{equation}

\ref{p:1xiii}: Set $A=L\proxcc{\gamma}B$. It follows from
\ref{p:1xii}, \eqref{e:moyo}, and \eqref{e:res} that
$0\in\zer(L^*\circ(\moyo{B}{\gamma})\circ L)\Leftrightarrow
0\in\zer(\moyo{A}{\gamma})\Leftrightarrow x\in J_{\gamma A}x
\Leftrightarrow 0\in Ax\Leftrightarrow x\in\zer A$.
\end{proof}

The following proposition shows that the resolvent of the operators
of Definition~\ref{d:1} can be computed explicitly.

\begin{proposition}
\label{p:2}
Let $L\in\BL(\HH,\GG)$, let $B\colon\GG\to2^\GG$, and let
$\gamma\in\RPP$. Then the following hold:
\begin{enumerate}
\item
\label{p:2i}
$J_{\gamma\brk1{L\proxc{\gamma}B}}=L^*\circ J_{\gamma B}\circ L$.
\item
\label{p:2ii}
$J_{\gamma\brk1{L\proxcc{\gamma}B}}=\Id_{\HH}-L^*\circ(\Id_{\GG}-
J_{\gamma B})\circ L$.
\end{enumerate}
\end{proposition}
\begin{proof}
\ref{p:2i}: By \eqref{e:res}, Proposition~\ref{p:1}\ref{p:1iv}, and
Definition~\ref{d:1}, 
\begin{equation}J_{\gamma\brk1{L\proxc{\gamma}B}}
=\brk2{\Id_{\HH}+\gamma\brk1{L\proxc{\gamma}B}}^{-1}
=\brk2{\Id_{\HH}+\brk1{L\proxc{}(\gamma B)}}^{-1}
=\brk1{L^*\pushfwd(\gamma B+\Id_{\GG})}^{-1}
=L^*\circ J_{\gamma B}\circ L.
\end{equation}

\ref{p:2ii}: By Proposition~\ref{p:1}\ref{p:1vi},
Definition~\ref{d:1}, \eqref{e:23}, and \ref{p:2i}, we obtain
\begin{equation}
J_{\gamma\brk1{L\proxcc{\gamma}B}}=J_{L\proxcc{}(\gamma B)}
=J_{\brk1{L\proxc{}(\gamma B)^{-1}}^{-1}}
=\Id_{\HH}-J_{L\proxc{}(\gamma B)^{-1}}
=\Id_{\HH}-L^*\circ(\Id_{\GG}-J_{\gamma B})\circ L,
\end{equation}
as claimed.
\end{proof}

The next result interprets resolvent compositions as parallel
compositions of perturbed operators.

\begin{proposition}
\label{p:3}
Let $L\in\BL(\HH,\GG)$, let $B\colon\GG\to2^\GG$, let
$\gamma\in\RPP$, and set $\Psi=\Id_{\GG}-L\circ L^*$.
Then the following hold:
\begin{enumerate}
\item
\label{p:3i}
$L\proxc{\gamma}B=L^*\pushfwd(B+\gamma^{-1}\Psi)$.
\item
\label{p:3ii}
$L\proxcc{\gamma}B=L^*\circ(B^{-1}+\gamma\Psi)^{-1}\circ L$.
\end{enumerate}
\end{proposition}
\begin{proof}
\ref{p:3i}: Combining Definition~\ref{d:1} and
Lemma~\ref{l:3}\ref{l:3v}, we deduce that
\begin{align}
L\proxc{\gamma}B
&=L^*\pushfwd\brk1{B+\gamma^{-1}\Id_{\GG}}-\gamma^{-1}\Id_{\HH}
\nonumber\\
&=L^*\pushfwd\brk1{B+\gamma^{-1}\Id_{\GG}
+L\circ(-\gamma^{-1}\Id_{\HH})\circ L^*}\nonumber\\
&=L^*\pushfwd\brk1{B+\gamma^{-1}\Psi}.
\end{align}

\ref{p:3ii}: It follows from Definition~\ref{d:1} and \ref{p:3i}
that
\begin{equation}
L\proxcc{\gamma}B=\brk1{L\proxc{1/\gamma}B^{-1}}^{-1}
=\brk1{L^*\pushfwd(B^{-1}+\gamma\Psi)}^{-1}
=L^*\circ(B^{-1}+\gamma\Psi)^{-1}\circ L,
\end{equation}
as announced.
\end{proof}

We proceed to provide particular instances in which the standard,
parallel, and resolvent compositions coincide.

\begin{proposition}
\label{p:4}
Let $L\in\BL(\HH,\GG)$, let $B\colon\GG\to2^\GG$, and let
$\gamma\in\RPP$. Then the following hold:
\begin{enumerate}
\item
\label{p:4i}
Suppose that $L$ is an isometry. Then
$L\proxc{\gamma}B=L\proxcc{\gamma}B$.
\item
\label{p:4ii}
Suppose that $L$ is a coisometry. Then $L\proxc{\gamma}B
=L^*\pushfwd B$ and $L\proxcc{\gamma}B=L^*\circ B\circ L$.
\item
\label{p:4iii}
Suppose that $L$ is invertible with $L^{-1}=L^*$. Then 
$L\proxc{\gamma}B=L^*\pushfwd B=L^*\circ B\circ L
=L\proxcc{\gamma}B$.
\end{enumerate}
\end{proposition}
\begin{proof}
\ref{p:4i}: Since $L^*\circ L=\Id_{\HH}$, we deduce from
Proposition~\ref{p:2} and \eqref{e:res} that
\begin{equation}
\gamma(L\proxcc{\gamma}B)
=\brk2{J_{\gamma\brk1{L\proxcc{\gamma}B}}}^{-1}-\Id_{\HH}
=\brk1{L^*\circ J_{\gamma B}\circ L}^{-1}-\Id_{\HH}
=\brk2{J_{\gamma\brk1{L\proxc{\gamma}B}}}^{-1}-\Id_{\HH}
=\gamma(L\proxc{\gamma}B).
\end{equation} 

\ref{p:4ii}: Since $L$ is a coisometry, 
$\Psi=\Id_{\GG}-L\circ L^*=0$. Therefore, we derive from
Proposition~\ref{p:3} that $L\proxc{\gamma}B=L^*\pushfwd B$ and
$L\proxcc{\gamma}B=L^*\circ B\circ L$.

\ref{p:4iii}: This follows from \ref{p:4i} and \ref{p:4ii}.
\end{proof}

\begin{corollary}
Let $L\in\BL(\HH,\GG)$, let $B\colon\GG\to2^\GG$, and let
$\gamma\in\RPP$. Let $\KK$ be a real Hilbert space and suppose that
$M\in\BL(\KK,\HH)$ and $S\in\BL(\GG)$ are coisometries. Then the
following hold:
\begin{enumerate}
\item
$M^*\pushfwd(L\proxc{\gamma}B)=(L\circ M)\proxc{\gamma}B$.
\item
$M^*\circ(L\proxcc{\gamma}B)\circ M=(L\circ M)\proxcc{\gamma}B$.
\item
$L\proxc{\gamma}(S^*\pushfwd B)=(S\circ L)\proxc{\gamma}B$.
\item
$L\proxcc{\gamma}(S^*\circ B\circ S)=(S\circ L)\proxcc{\gamma}B$
\end{enumerate}
\end{corollary}
\begin{proof}
Combine Proposition~\ref{p:4}\ref{p:4ii} and
Proposition~\ref{p:1}\ref{p:1ix}.
\end{proof}

Now, we present several examples of resolvent compositions and
cocompositions, starting with the representation of the Yosida
approximation as one such composition.

\begin{example}
\label{ex:yosi}
Let $A\colon\HH\to2^\HH$ and $\gamma\in\RPP$. Set $L=\Id_{\HH}/2$
and $B=2A(2\Id_{\HH})$. Then
$L\proxcc{\gamma/3}B=\moyo{A}{\gamma}$. 
\end{example}
\begin{proof}
It follows from Proposition~\ref{p:3}\ref{p:3ii} that
$L\proxcc{4\gamma/3}A
=(1/2)(A^{-1}+\gamma\Id_{\HH})^{-1}(\Id_{\HH}/2)
=(1/2)\moyo{A}{\gamma}(\Id_{\HH}/2)$. Therefore,
Proposition~\ref{p:1}\ref{p:1vi}--\ref{p:1vii} implies that
$\moyo{A}{\gamma}=2(L\proxcc{4\gamma/3}A)(2\Id_{\HH})
=L\proxcc{\gamma/3}B$, as claimed.
\end{proof}

\begin{example}
Let $V$ be a closed vector subspace of $\HH$, $L\in\BL(\HH,\GG)$,
$B\colon\GG\to2^\GG$, and $\gamma\in\RPP$. Suppose that $L$ is
surjective and that $L^*\circ L=\proj_V$. Then
$L\proxc{\gamma}B=L^*\pushfwd B$,
$L\proxcc{\gamma}B=L^*\circ B\circ L$, and
$\moyo{(L^*\circ B\circ L)}{\gamma}
=L^*\circ(\moyo{B}{\gamma})\circ L$.
\end{example}
\begin{proof}
In this case, $L$ is a coisometry. Therefore,
Proposition~\ref{p:4}\ref{p:4ii} implies that
$L\proxc{\gamma}B=L^*\pushfwd B$ and $L\proxcc{\gamma}B=L^*\circ
B\circ L$. Further, we use Proposition~\ref{p:1}\ref{p:1xii} to
deduce that $\moyo{(L^*\circ B\circ L)}{\gamma}
=\moyo{(L\proxcc{\gamma}B)}{\gamma}
=L^*\circ(\moyo{B}{\gamma})\circ L$, which completes the proof.
\end{proof}

\begin{example}
\label{ex:37}
Let $S\in\BL(\HH,\GG)$, let $A\colon\GG\to2^\GG$, let 
$\gamma\in\RPP$, and let $\mu\in\RPP$. Suppose that 
$S\circ S^*=\mu\Id_{\GG}$. Then the
following hold:
\begin{enumerate}
\item
\label{ex:37i}
Set $L=S/\sqrt{\mu}$ and $B=\sqrt{\mu}A(\sqrt{\mu}\Id_{\GG})$. Then
$L\proxcc{\gamma}B=S^*\circ A\circ S$.
\item
\label{ex:37ii}
$J_{\gamma S^*\circ A\circ S}
=\Id_{\HH}-\mu^{-1}S^*\circ(\Id_{\GG}-J_{\mu\gamma A})\circ S$.
\end{enumerate}
\end{example}
\begin{proof}
\ref{ex:37i}: In this case, $L$ is a coisometry and
Proposition~\ref{p:4}\ref{p:4ii} yields
$L\proxcc{\gamma}B=L^*\circ B\circ L=S^*\circ A\circ S$.

\ref{ex:37ii}: By \ref{ex:37i}, $S^*\circ A\circ S
=L\proxcc{\gamma}B$. Therefore, the result follows from 
Proposition~\ref{p:2}\ref{p:2ii} and basic resolvent calculus.
\end{proof}

\begin{example}
\label{ex:39}
Suppose that $L\in\BL(\HH,\GG)$ satisfies $\|L\|<1$, let
$B\colon\GG\to2^\GG$, and let $\gamma\in\RPP$. Let 
$\XX=\HH\oplus\GG$, set
$\Psi=\Id_{\GG}-L\circ L^*$, and set
$L_\Psi\colon\XX\to\GG\colon(x,y)\mapsto Lx+\Psi^{1/2}y$. Then
\begin{equation}
\label{e:ex39a}
L_\Psi\proxcc{\gamma}B\colon\XX\to2^\XX\colon(x,y)\mapsto
\brk2{L^*\brk1{B(Lx+\Psi^{1/2}y)}}\times
\brk2{\Psi^{1/2}\brk1{B(Lx+\Psi^{1/2}y)}}.
\end{equation}
\end{example}
\begin{proof}
Note that $\Psi$ is self-adjoint and that 
\begin{equation}
(\forall y\in\GG)\quad\scal{\Psi y}{y}_{\GG}
=\|y\|^2_{\GG}-\|L^*y\|^2_{\HH}\geq(1-\|L\|^2)\|y\|^2_{\GG}.
\end{equation}
Thus, $\Psi$ is strongly monotone and $\Psi^{1/2}$ is well
defined. Further, since
$L_\Psi^*\colon\GG\to\XX\colon y\mapsto(L^*y,\Psi^{1/2}y)$, we
deduce that
\begin{equation}
(\forall y\in\GG)\quad
L_\Psi(L_\Psi^*y)=L_\Psi(L^*y,\Psi^{1/2}y)=L(L^*y)+\Psi y=y.
\end{equation}
Therefore, $L_\Psi$ is a coisometry, and it follows from 
Proposition~\ref{p:4}\ref{p:4ii} that
$L_\Psi\proxcc{\gamma}B=L_\Psi^*\circ B\circ L_\Psi$, which
establishes \eqref{e:ex39a}.
\end{proof}

\begin{example}[resolvent mixture]
\label{ex:mix}
Let $0\neq p\in\NN$ and let $\gamma\in\RPP$. For every
$k\in\{1,\ldots,p\}$, let $\GG_k$ be a real Hilbert space, let
$L_k\in\BL(\HH,\GG_k)$, let $B_k\colon\GG_k\to2^{\GG_k}$, and let
$\alpha_k\in\RPP$.  Let $\GG=\bigoplus_{k=1}^p\GG_k$, and set
\begin{equation}
\label{e:mix3}
L\colon\HH\to\GG\colon x\mapsto
\brk1{\sqrt{\alpha_1}L_1x,\ldots,\sqrt{\alpha_p}L_px}
\end{equation}
and
\begin{equation}
\label{e:mix4}
B\colon\GG\to2^\GG\colon(y_k)_{1\leq k\leq p}\mapsto
\brk2{\sqrt{\alpha_1}B_1\brk1{y_1/\sqrt{\alpha_1}}}\times\cdots
\times\brk2{\sqrt{\alpha_p}B_p\brk1{y_p/\sqrt{\alpha_p}}}.
\end{equation}
Then Proposition~\ref{p:2} yields
\begin{equation}
\label{e:mix1}
J_{\gamma\brk1{L\proxc{\gamma}B}}
=\sum_{k=1}^p\alpha_kL_k^*\circ J_{\gamma B_k}\circ L_k
\end{equation}
and
\begin{equation}
\label{e:mix2}
J_{\gamma\brk1{L\proxcc{\gamma}B}}
=\Id_{\HH}-\sum_{k=1}^p\alpha_kL_k^*\circ
(\Id_{\GG_k}-J_{\gamma B_k})\circ L_k.
\end{equation}
The operators
\begin{equation}
\Rm{\gamma}(B_k,L_k)_{1\leq k\leq p}
=L\proxc{\gamma}B
=\brk3{\sum_{k=1}^p\alpha_kL_k^*\circ\brk1{B_k
+\gamma^{-1}\Id_{\GG_k}}^{-1}\circ L_k}^{-1}-\gamma^{-1}\Id_{\HH}
\end{equation}
and
\begin{equation}
\Rcm{\gamma}(B_k,L_k)_{1\leq k\leq p}
=L\proxcc{\gamma}B
=\brk4{\brk3{\sum_{k=1}^p\alpha_kL_k^*\circ\brk1{B_k^{-1}
+\gamma\Id_{\GG_k}}^{-1}\circ L_k}^{-1}
-\gamma\Id_{\HH}}^{-1}
\end{equation}
are called the {\em resolvent mixture} and {\em comixture},
respectively, and were initially introduced in
\cite[Example~3.4]{Svva23} (see also \cite{Jota24} for further
developments).
\end{example}

\begin{example}[resolvent average]
\label{ex:rave}
In the context of Example~\ref{ex:mix}, suppose that 
$\sum_{k=1}^p\alpha_k=1$ and that, for every
$k\in\{1,\ldots,p\}$, $\GG_k=\HH$ and $L_k=\Id_{\HH}$. 
Since $L^*\colon\GG\to\HH\colon(y_k)_{1\leq k\leq p}\mapsto
\sum_{k=1}^p\sqrt{\alpha_k}y_k$, the linear operator $L$ is an
isometry. Thus, Proposition~\ref{p:4}\ref{p:4i} yields
$L\proxc{\gamma}B=L\proxcc{\gamma}B$. This operator is called the
\emph{resolvent average} of $(B_k)_{1\leq k\leq p}$ and
$(\alpha_k)_{1\leq k\leq p}$, denoted by
$\rav_{\gamma}(B_k,\alpha_k)_{1\leq k\leq p}$, which has been
studied in \cite{Baus16,Livre1,Jota24,Svva23,Wang11}, namely,
\begin{equation}
L\proxc{\gamma}B
=\brk3{\sum_{k=1}^p\alpha_k
\brk1{B_k+\gamma^{-1}\Id_{\HH}}^{-1}}^{-1}-\gamma^{-1}\Id_{\HH}
=\rav_{\gamma}(B_k,\alpha_k)_{1\leq k\leq p}.
\end{equation}
In addition, \eqref{e:mix1} yields
$J_{\gamma\rav_{\gamma}(B_k,\alpha_k)_{1\leq k\leq p}}
=\sum_{k=1}^p\alpha_kJ_{\gamma B_k}$.
\end{example}

\section{Monotonicity of resolvent compositions}
\label{sec:4}

We leverage the results of Section~\ref{sec:3} to establish
conditions that ensure the monotonicity of resolvent compositions.

\begin{proposition}
\label{p:20}
Suppose that $0\neq L\in\BL(\HH,\GG)$, let $B\colon\GG\to2^\GG$,
let $\gamma\in\RPP$, let $\alpha\in\intv[r]{-1/\gamma}{\pinf}$, and
set $\beta=(\alpha+\gamma^{-1})\|L\|^{-2}-\gamma^{-1}$. Suppose
that $B-\alpha\Id_{\GG}$ is monotone. Then
$L\proxc{\gamma}B-\beta\Id_{\HH}$ is monotone. 
\end{proposition}
\begin{proof}
Set $\Psi=\Id_{\GG}-L\circ L^*$ and 
$\mathcal{M}=(B-\alpha\Id_{\GG})
+(\alpha+\gamma^{-1})\|L\|^{-2}(\|L\|^2\Id_{\GG}-L\circ L^*)$.
It follows from Proposition~\ref{p:3}\ref{p:3i} and
Lemma~\ref{l:3}\ref{l:3v} that
\begin{align}
\label{e:p20}
\brk1{L\proxc{\gamma}B}-\beta\Id_{\HH}
&=L^*\pushfwd\brk1{B+\gamma^{-1}\Psi}-\beta\Id_{\HH}
\nonumber\\
&=L^*\pushfwd\brk1{B+\gamma^{-1}\Psi
+L\circ(-\beta\Id_{\HH})\circ L^*}\nonumber\\
&=L^*\pushfwd\brk1{
(B-\alpha\Id_{\GG})
+(\alpha+\gamma^{-1})\|L\|^{-2}(\|L\|^2\Id_{\GG}-L\circ L^*)}
\nonumber\\
&=L^*\pushfwd\mathcal{M}.
\end{align}
Since $\alpha+\gamma^{-1}\geq 0$ and the operators 
$B-\alpha\Id_{\GG}$ and $\|L\|^2\Id_{\GG}-L\circ L^*$ are
monotone, \cite[Propositions~20.10]{Livre1} implies
that $\mathcal{M}$ is monotone. Therefore, the assertion follows
from \eqref{e:p20} and \cite[Proposition~25.41(ii)]{Livre1}.
\end{proof}

\begin{corollary}
\label{c:42}
Suppose that $L\in\BL(\HH,\GG)$ satisfies $0<\|L\|\leq 1$, let
$B\colon\GG\to2^\GG$, let $\gamma\in\RPP$, let $\alpha\in\RPP$, and
set $\beta=(\alpha+\gamma^{-1})\|L\|^{-2}-\gamma^{-1}$. Suppose
that $B$ is $\alpha$-strongly monotone. Then $L\proxc{\gamma}B$ is
$\beta$-strongly monotone.
\end{corollary}
\begin{proof}
Since $\beta=\alpha\|L\|^{-2}+\gamma^{-1}(\|L\|^{-2}-1)>0$, the
conclusion follows from Proposition~\ref{p:20}.
\end{proof}

\begin{corollary}
\label{c:43}
Suppose that $L\in\BL(\HH,\GG)$ satisfies $0<\|L\|<1$, let
$B\colon\GG\to2^\GG$ be monotone, let $\gamma\in\RPP$, and set
$\beta=\gamma^{-1}(\|L\|^{-2}-1)$. Then $L\proxc{\gamma}B$ is
$\beta$-strongly monotone.
\end{corollary}
\begin{proof}
This follows from Proposition~\ref{p:20} when $\alpha=0$.
\end{proof}

\begin{proposition}
\label{p:10}
Let $L\in\BL(\HH,\GG)$, $B\colon\GG\to2^\GG$, and
$\gamma\in\RPP$. Suppose that $B+\gamma^{-1}(\Id_{\GG}-L\circ L^*)$
is monotone. Then the following hold:
\begin{enumerate}
\item
\label{p:10i}
$L\proxc{\gamma}B$ is monotone.
\item
\label{p:10ii}
Suppose that $\ran L\subset\ran(\Id_{\GG}+\gamma B)$. Then
$L\proxc{\gamma}B$ is maximally monotone.
\end{enumerate}
\end{proposition}
\begin{proof}
Set $\Psi=\Id_{\GG}-L\circ L^*$ and recall from
Proposition~\ref{p:3}\ref{p:3i} that
$L\proxc{\gamma}B=L^*\pushfwd(B+\gamma^{-1}\Psi)$.

\ref{p:10i}: Since $B+\gamma^{-1}\Psi$ is monotone, we deduce from
\cite[Proposition~25.41(ii)]{Livre1} that $L\proxc{\gamma}B$ is
monotone.

\ref{p:10ii}: Since monotonicity is preserved under multiplication
by positive scalars, $\gamma(L\proxc{\gamma}B)$ is monotone by
\ref{p:10i}. Further, by assumption, 
$\ran L\subset\ran(\Id_{\GG}+\gamma B)
=\dom(\Id_{\GG}+\gamma B)^{-1}=\dom J_{\gamma B}$. Therefore,
$\dom(L^*\circ J_{\gamma B}\circ L)=\dom(J_{\gamma B}\circ L)=\HH$,
and it follows from Proposition~\ref{p:2}\ref{p:2i} that
$\ran(\Id_{\HH}+\gamma(L\proxc{\gamma}B))
=\dom J_{\gamma(L\proxc{\gamma}B)}=\HH$. Altogether, we deduce from
\cite[Theorem~21.1 (Minty)]{Livre1} that $\gamma(L\proxc{\gamma}B)$
is maximally monotone. Hence, $L\proxc{\gamma}B$ is maximally
monotone.
\end{proof}

The following result recovers \cite[Proposition~4.4(i)--(ii) and
Theorem~4.5(i)--(ii)]{Svva23}, which were proven when $\gamma=1$
using distinct approaches. 

\begin{corollary}
\label{c:15}
Suppose that $L\in\BL(\HH,\GG)$ satisfies $\|L\|\leq 1$, let
$B\colon\GG\to2^\GG$ be monotone, and let $\gamma\in\RPP$.
Then the following hold:
\begin{enumerate}
\item
\label{c:15i}
$L\proxc{\gamma}B$ and
$L\proxcc{\gamma}B$ are monotone.
\item
\label{c:15ii}
Suppose that $B$ is maximally monotone. Then $L\proxc{\gamma}B$ and
$L\proxcc{\gamma}B$ are maximally monotone.
\end{enumerate}
\end{corollary}
\begin{proof} Set $\Psi=\Id_{\GG}-L\circ L^*$. Since $\|L\|\leq 1$,
$\Psi$ is monotone. Thus, \cite[Proposition~20.10]{Livre1} implies
that $B+\gamma^{-1}\Psi$ and $B^{-1}$ are monotone.

\ref{c:15i}: By Proposition~\ref{p:10}\ref{p:10i},
$L\proxc{\gamma}B$ and $L\proxc{1/\gamma}B^{-1}$ are monotone.
Therefore, we combine Definition~\ref{d:1} and
\cite[Proposition~20.10]{Livre1} to deduce that
$L\proxcc{\gamma}B=(L\proxc{1/\gamma}B^{-1})^{-1}$ is monotone.

\ref{c:15ii}: By \cite[Proposition~20.22]{Livre1}, $\gamma B$ and
$B^{-1}$ are maximally monotone. Therefore,
\cite[Theorem~21.1]{Livre1} yields 
$\ran(\Id_{\GG}+\gamma B)=\GG$, and, by
Proposition~\ref{p:10}\ref{p:10ii}, $L\proxc{\gamma}B$ is maximally
monotone. Similarly, $L\proxc{1/\gamma}B^{-1}$ is maximally
monotone, and it follows from \cite[Proposition~20.22]{Livre1} that
$L\proxcc{\gamma}B=(L\proxc{1/\gamma}B^{-1})^{-1}$ is maximally
monotone as well.
\end{proof}

Next, we state several instantiations of resolvent compositions
which are monotone.

\begin{example}
\label{ex:mix2}
Let $0\neq p\in\NN$ and let $\gamma\in\RPP$. For every
$k\in\{1,\ldots,p\}$, let $\GG_k$ be a real Hilbert space, let
$L_k\in\BL(\HH,\GG_k)$, let $B_k\colon\GG_k\to2^{\GG_k}$ be
maximally monotone, and let $\alpha_k\in\RPP$. Suppose that
$\sum_{k=1}^p\alpha_k\|L_k\|^2\leq 1$. Then
$\Rm{\gamma}(B_k,L_k)_{1\leq k\leq p}$ and
$\Rcm{\gamma}(B_k,L_k)_{1\leq k\leq p}$ are maximally monotone.
\end{example}
\begin{proof}
Define $L$ as in \eqref{e:mix3} and $B$ as in \eqref{e:mix4}. Thus,
the assumption $\sum_{k=1}^p\alpha_k\|L_k\|^2\leq 1$ implies that
$\|L\|\leq 1$, and, by \cite[Proposition~23.18]{Livre1}, $B$ is
maximally monotone. Since 
$\Rm{\gamma}(B_k,L_k)_{1\leq k\leq p}=L\proxc{\gamma}B$ and
$\Rcm{\gamma}(B_k,L_k)_{1\leq k\leq p}=L\proxcc{\gamma}B$, the
conclusion follows from Corollary~\ref{c:15}\ref{c:15ii}.
\end{proof}

\begin{example}
\label{ex:dr}
Let $A_1\colon\HH\to2^\HH$ and $A_2\colon\HH\to2^\HH$ be maximally
monotone, let $\GG=\HH\oplus\HH$, set $L\colon\HH\to\GG\colon
x\mapsto(x,-x)$, and set
\begin{equation}
B\colon\GG\to2^\GG\colon(x,y)\mapsto
\brk1{A_1x}\times\brk1{A_2^{-1}y-2x}.
\end{equation}
Then $L\proxc{}B$ is maximally monotone and
\begin{equation}
\label{e:dr}
J_{L\proxc{}{B}}
=\dfrac{1}{2}\Id_{\HH}
+\dfrac{1}{2}\brk1{2J_{A_2}-\Id_{\HH}}\circ
\brk1{2J_{A_1}-\Id_{\HH}}.
\end{equation}
In other words, the resolvent of $L\proxc{}B$ is the 
{\em Douglas--Rachford splitting operator} of $A_2$ and $A_1$ 
(see \cite{Livre1,Ecks92}).
\end{example}
\begin{proof}
Set $\Psi=\Id_{\GG}-L\circ L^*$,
$\mathcal{M}\colon\GG\to2^\GG\colon(x,y)\mapsto 
(A_1x)\times(A_2^{-1}y)$, and $S\colon\GG\to\GG\colon(x,y)\mapsto
(y,-x)$. Note that $\Psi\colon(x,y)\mapsto(y,x)$, that $S$ is
monotone, and that, since $A_1$ and $A_2$ are monotone,
$\mathcal{M}$ is monotone as well. Thus,
\cite[Proposition~20.10]{Livre1} implies that
$B+\Psi=\mathcal{M}+S$ is monotone. Further, it is straightforward
verify that
\begin{equation}
\label{e:exdr}
\brk1{\forall(x,y)\in\GG}\quad
J_B(x,y)=\brk1{J_{A_1}x,J_{A_2^{-1}}(y+2J_{A_1}x)}.
\end{equation}
Therefore, it follows from Proposition~\ref{p:10}\ref{p:10ii} that
$L\proxc{}B$ is maximally monotone. In addition, since
$L^*\colon\GG\to\HH\colon(x,y)\mapsto x-y$,
Proposition~\ref{p:2}\ref{p:2i}, \eqref{e:exdr}, and \eqref{e:23}
yield
\begin{align}
(\forall x\in\HH)\quad
J_{L\proxc{}B}x&=L^*\brk1{J_B(x,-x)}\nonumber\\
&=L^*\brk1{J_{A_1}x,J_{A_2^{-1}}(2J_{A_1}x-x)}\nonumber\\
&=J_{A_1}x-J_{A_2^{-1}}(2J_{A_1}x-x)\nonumber\\
&=J_{A_1}x-\brk1{2J_{A_1}x-x-J_{A_2}(2J_{A_1}x-x)}\nonumber\\
&=\dfrac{1}{2}x+\dfrac{1}{2}\brk1{2J_{A_2}-\Id_{\HH}}(
2J_{A_1}x-x),
\end{align}
which establishes \eqref{e:dr}.
\end{proof}

\begin{remark}
Consider the setting of Example~\ref{ex:dr}. Then the operator $B$
is not necessarily monotone (take $A_1=0$ and $A_2=N_{\{0\}}$)
and the norm of the linear operator is greater than $1$
($\|L\|=\sqrt{2}$). As a result, the resolvent composition 
$L\proxc{}B$ can be maximally monotone, even in cases when $B$ is
not monotone and $\|L\|>1$.
\end{remark}

The following example recovers the operator used in \cite{Mali23}
for finding a zero in the sum of $p\geq 2$ maximally monotone
operators. For the sake of simplicity, we represent operators using
matrices. 

\begin{example}
Let $\gamma\in\RPP$, let $\KK$ be a real Hilbert space, let
$p\in\NN\setminus\{0,1\}$, and let $(A_k)_{1\leq k\leq p}$ be a
family of maximally monotone operators in $\KK$. Let
$\HH=\bigoplus_{k=1}^{p-1}\KK$, let $\GG=\bigoplus_{k=1}^p\KK$,
set
\begin{equation}
L=\begin{bmatrix}
\Id_{\KK}\\
-\Id_{\KK}&\Id_{\KK}\\
&\ddots&\ddots&\\
&&-\Id_{\KK}&\Id_{\KK}\\
&&&-\Id_{\KK}
\end{bmatrix}\in\BL(\HH,\GG),
\end{equation}
and set
\begin{equation}
B=\gamma\begin{bmatrix}
A_1\\
-\Id_{\KK}&A_2\\
&\ddots&\ddots\\
&&-\Id_{\KK}&A_{p-1}\\
-\Id_{\KK}&&&-\Id_{\KK}&A_p
\end{bmatrix}\colon\GG\to2^\GG.
\end{equation}
Then, for every $\boldsymbol{z}=(z_k)_{1\leq k\leq p-1}\in\HH$, 
\begin{equation}
\label{e:300} 
J_{\gamma\brk2{(\gamma L)\proxcc{\gamma}B^{-1}}}\boldsymbol{z}
=\boldsymbol{z}+\gamma^2
\begin{bmatrix}
x_2-x_1\\
x_3-x_2\\
\vdots\\
x_p-x_{p-1}
\end{bmatrix},
\quad\text{where}\quad
\begin{cases}
x_1=J_{A_1}z_1,\\
(\forall k\in\{2,\ldots,p-1\})\,\,
x_k=J_{A_k}(z_k+x_{k-1}-z_{k-1}),\\
x_p=J_{A_p}(x_1+x_{p-1}-z_{p-1}).
\end{cases}
\end{equation}
\end{example}
\begin{proof}
Let $\boldsymbol{z}=(z_k)_{1\leq k\leq p-1}\in\HH$. It is
straightforward verify that $J_{\gamma^{-1}B}(L\boldsymbol{z})
=(x_k)_{1\leq k\leq p}=\boldsymbol{x}$. On the other hand, recall
from \cite[Proposition~23.20]{Livre1} that
$\Id_{\GG}-J_{\gamma B^{-1}}
=\gamma J_{\gamma^{-1}B}(\gamma^{-1}\Id_{\GG})$. Further, note that
\begin{equation}
\label{e:310}
L^*\boldsymbol{x}=
\begin{bmatrix}
x_1-x_2\\
x_2-x_3\\
\vdots\\
x_{p-1}-x_p
\end{bmatrix}
\end{equation}
 Altogether, we
deduce from Proposition~\ref{p:2}\ref{p:2ii} and \eqref{e:310} that
\begin{equation}
J_{\gamma\brk2{(\gamma L)\proxcc{\gamma}B^{-1}}}\boldsymbol{z}
=\boldsymbol{z}-(\gamma L)^*\brk1{
\gamma J_{\gamma^{-1}B}(L\boldsymbol{z})}
=\boldsymbol{z}-\gamma^2L^*\boldsymbol{x}
=\boldsymbol{z}+\gamma^2
\begin{bmatrix}
x_2-x_1\\
x_3-x_2\\
\vdots\\
x_p-x_{p-1},
\end{bmatrix},
\end{equation}
which establishes \eqref{e:300}.
\end{proof}

\begin{remark}
As shown in \cite[Lemma~3]{Mali23}, the operator given in
\eqref{e:300} is $\gamma^2$-averaged when $\gamma\in\zeroun$.
\end{remark}

\begin{example}
\label{ex:1}
Let $L\in\BL(\HH,\GG)$ be such that $\ran L$ is closed, let  
$B\colon\GG\to2^\GG$ be maximally monotone, and let
$\gamma\in\RPP$. Let $\XX$ be the real Hilbert space obtained by
endowing $\HH$ with the scalar product
\begin{equation}
\label{e:ex1}
\scal{\cdot}{\cdot}_{\XX}\colon\XX\times\XX\to\RR\colon(x,y)
\mapsto\scal{Lx}{Ly}_{\GG}+\scal{x}{\proj_{\ker L}y}_{\HH},
\end{equation}
and set $L_{\XX}\colon\XX\to\GG\colon x\mapsto Lx$. Then the
following hold: 
\begin{enumerate}
\item
\label{ex:1i}
$L_{\XX}\proxc{\gamma}B$ and $L_{\XX}\proxcc{\gamma}B$ are
maximally monotone.
\item
\label{ex:1ii}
$J_{\gamma\brk1{L_{\XX}\proxc{\gamma}B}}
=L^\dagger\circ J_{\gamma B}\circ L$.
\item
\label{ex:1iii}
$J_{\gamma\brk1{L_{\XX}\proxcc{\gamma}B}}
=\Id_{\XX}-L^\dagger\circ(\Id_{\GG}-J_{\gamma B})\circ L$.
\item
\label{ex:1iv}
Suppose that $\ker L=\{0\}$. Then 
$L_{\XX}\proxc{\gamma}B=L_{\XX}\proxcc{\gamma}B$.
\item
\label{ex:1v}
Suppose that $\ran L=\GG$. Then
$L_{\XX}\proxc{\gamma}B=L^\dagger\pushfwd B$ and
$L_{\XX}\proxcc{\gamma}B=L^\dagger\circ B\circ L$.
\end{enumerate}
\end{example}
\begin{proof}
Let $x\in\HH$ and $y\in\GG$. It follows from \eqref{e:ex1} that
$\|L_{\XX}\|\leq 1$ since
\begin{equation}
(\forall z\in\HH)\quad
\|L_{\XX}z\|_{\GG}^2=\|Lz\|_{\GG}^2
\leq\|Lz\|_{\GG}^2+\|\proj_{\ker L}z\|_{\HH}^2
=\|z\|_{\XX}^2.
\end{equation}
Further, the identities $L^*y=L^*(L(L^\dagger y))$
and $L^\dagger y\in(\ker L)^\bot$
\cite[Proposition~3.30(i)]{Livre1} imply that
\begin{align}
\scal{L_{\XX}x}{y}_{\GG}&=\scal{x}{L^*y}_{\HH}\nonumber\\
&=\scal{x}{L^*\brk1{L(L^\dagger y)}}_{\HH}\nonumber\\
&=\scal{Lx}{L(L^\dagger y)}_{\GG}
+\scal{x}{\proj_{\ker L}(L^\dagger y)}_{\HH}\nonumber\\
&=\scal{x}{L^\dagger y}_{\XX}.
\end{align}
In turn, $L_{\XX}^*\colon\GG\to\XX\colon y\mapsto L^\dagger y$.

\ref{ex:1i}: A consequence of Corollary~\ref{c:15}\ref{c:15ii}.

\ref{ex:1ii}--\ref{ex:1iii}: A consequence of
Proposition~\ref{p:2}.

\ref{ex:1iv}: By \cite[Proposition~Corollary~3.32(iv)]{Livre1},
$L^\dagger\circ L=\Id_{\HH}$. Therefore, $L_{\XX}$ is an isometry,
and the assertion follows from Proposition~\ref{p:4}\ref{p:4i}.

\ref{ex:1v}: By \cite[Proposition~3.30(ii)]{Livre1}, 
$L\circ L^\dagger=\Id_{\GG}$. Therefore, $L_{\XX}$ is a coisometry,
and the assertion follows from Proposition~\ref{p:4}\ref{p:4ii}.
\end{proof}

\begin{remark}
When $L\in\BL(\HH,\GG)$ satisfies $L^\dagger\circ L=\Id_{\HH}$,
the operator
$T=L^\dagger\circ J_{\gamma\partial\|\cdot\|_{\HH}}\circ L$
has been used to enhance the performance of wavelet-domain 
denoising \cite{Coif95}. Consequently,
Example~\ref{ex:1}\ref{ex:1ii} shows that this method implicitly
involves resolvent compositions.
\end{remark}

\begin{example}
\label{ex:20}
Let $U\in\BL(\HH)$ be a self-adjoint and strongly monotone
operator. In the context of Example~\ref{ex:1}, assume that
$\GG=\HH$ and that $L=U^{-1/2}$. Then 
\begin{equation}
L_{\XX}\proxc{\gamma}B
=U^{1/2}\circ B\circ U^{-1/2}.
\end{equation}
\end{example}
\begin{proof}
In this case, $L$ is invertible and $L^\dagger=L^{-1}=U^{1/2}$.
Therefore, Example~\ref{ex:1}\ref{ex:1iv}--\ref{ex:1v} implies
that $L_{\XX}\proxc{\gamma}B=L_{\XX}\proxcc{\gamma}B
=U^{1/2}\circ B\circ U^{-1/2}$, as claimed.
\end{proof}

\begin{example}
\label{ex:proxc}
Suppose that $L\in\BL(\HH,\GG)$ satisfies $0<\|L\|\leq 1$, let
$g\colon\GG\to\RX$ be a proper function that admits a continuous
affine minorant, let $\gamma\in\RPP$, let
\begin{equation}
L\proxc{\gamma}g
=\brk2{\moyo{\brk1{g^*}}{\frac{1}{\gamma}}\circ L}^*
-\frac{1}{2\gamma}\|\cdot\|_{\HH}^2
\end{equation}
be the \emph{proximal composition} of $g$ and $L$, and let
$L\proxcc{\gamma}g=(L\proxc{1/\gamma}g^*)^*$ be the \emph{proximal
cocomposition} of $g$ and $L$ (see \cite{Jota24,Svva23,Eect24}).
Then the following hold:
\begin{enumerate}
\item
\label{ex:proxci}
$L\proxc{\gamma}\partial g^{**}=\partial(L\proxc{\gamma}g)$.
\item
\label{ex:proxcii}
$L\proxcc{\gamma}\partial g^{**}=\partial(L\proxcc{\gamma}g)$.
\end{enumerate}
\end{example}
\begin{proof}
Set $\Psi=\Id_{\GG}-L\circ L^*$ and recall from
\cite[Lemma~2.1(v)]{Eect24} that $g^*\in\Gamma_0(\GG)$.

\ref{ex:proxci}: By \cite[Proposition~3.11(i)]{Eect24} and
Proposition~\ref{p:3}\ref{p:3i},
$\partial(L\proxc{\gamma}g)
=L^*\pushfwd(\partial g^{**}+\gamma^{-1}\Psi)
=L\proxc{\gamma}\partial g^{**}$.

\ref{ex:proxcii}: Note that \eqref{e:29} and the identity
$g^{***}=g^*$ yield $(\partial g^{**})^{-1}=\partial g^{***}
=\partial g^*$. Therefore, it follows from
\cite[Proposition~3.11(ii)]{Eect24} and
Proposition~\ref{p:3}\ref{p:3ii} that
$\partial(L\proxcc{\gamma}g)
=L^*\circ(\partial g^*+\gamma\Psi)^{-1}\circ L
=L\proxcc{\gamma}\partial g^{**}$, which provides the desired
identity.
\end{proof}

We conclude this section by examining the resolvent composition of
uniformly monotone operators as well as the Fitzpatrick function of 
resolvent compositions. 

\begin{proposition}
\label{p:41}
Suppose that $L\in\BL(\HH,\GG)$ satisfies $0<\|L\|\leq 1$, let
$B\colon\GG\to2^\GG$ be maximally monotone, and let
$\gamma\in\RPP$. Suppose that $B$ is uniformly monotone with
modulus $\phi\colon\RP\to\RPX$, i.e.,
$\phi$ is increasing, vanishes only at $0$, and
\begin{equation}
\brk1{\forall(y_1,y_1^*)\in\gra B}\brk1{\forall(y_2,y_2^*)\in
\gra B}\quad
\scal{y_1-y_2}{y_1^*-y_2^*}_{\GG}\geq\phi(\|y_1-y_2\|_{\GG}),
\end{equation}
and set $\phi_L=L\proxc{\gamma/2}(\phi\circ\|\cdot\|_{\GG})$. Then
\begin{equation}
\brk1{\forall(x_1,x_1^*)\in\gra(L\proxc{\gamma}B)}
\brk1{\forall(x_2,x_2^*)\in\gra(L\proxc{\gamma}B)}\quad
\scal{x_1-x_2}{x_1^*-x_2^*}_{\GG}\geq\phi_L(x_1-x_2).
\end{equation}
\end{proposition}
\begin{proof}
Note that $\phi\circ\|\cdot\|_{\GG}\geq 0$ and that
$\phi(\|0\|_{\GG})=0$. Thus, $\phi\circ\|\cdot\|_{\GG}$ is a proper
function minorized by the affine function $0$. Further, by
\cite[Proposition~13.16]{Livre1}, 
$\phi\circ\|\cdot\|_{\GG}\geq(\phi\circ\|\cdot\|_{\GG})^{**}$.
On the other hand, recall from Corollary~\ref{c:15}\ref{c:15ii}
that $L\proxc{\gamma}B$ is maximally monotone. Let
$(x_1,x_1^*)\in\gra(L\proxc{\gamma}B)$ and 
$(x_2,x_2^*)\in\gra(L\proxc{\gamma}B)$. It follows from
Proposition~\ref{p:2}\ref{p:2i} that
\begin{align}
\label{e:mod}
(\forall k\in\{1,2\})\quad
x_k^*\in\brk1{L\proxc{\gamma}B}x_k&\Leftrightarrow
J_{\gamma\brk1{L\proxc{\gamma}B}}(x_k+\gamma x_k^*)=x_k\nonumber\\
&\Leftrightarrow
\begin{cases}
(\exi p_k\in\GG)\,\,L^*p_k=x_k\\
J_{\gamma B}\brk1{L(x_k+\gamma x_k^*)}=p_k
\end{cases}\nonumber\\
&\Leftrightarrow
\begin{cases}
(\exi p_k\in\GG)\,\,L^*p_k=x_k\\
\brk1{p_k,L(\gamma^{-1}x_k+x_k^*)-\gamma^{-1}p_k}\in\gra B.
\end{cases}
\end{align}
Since $B$ is uniformly monotone with modulus $\phi$, we deduce that
\begin{align}
\label{e:mod2}
\gamma^{-1}\|x_1-x_2\|_{\HH}^2+\scal{x_1-x_2}{x_1^*-x_2^*}_{\HH}
-\gamma^{-1}\|p_1-p_2\|_{\GG}^2={}
&\scal{p_1-p_2}{\gamma^{-1}L(x_1-x_2)+L(x_1^*-x_2^*)}_{\GG}
\nonumber\\
&-\gamma^{-1}\scal{p_1-p_2}{p_1-p_2}_{\GG}
\nonumber\\
\geq{}&\phi(\|p_1-p_2\|_{\GG}).
\end{align}
Therefore, since $L^*(p_1-p_2)=x_1-x_2$, we deduce from
\eqref{e:mod2} and \cite[Proposition~3.2(i)]{Eect24} that
\begin{align}
\scal{x_1-x_2}{x_1^*-x_2^*}_{\HH}&\geq
\phi(\|p_1-p_2\|_{\GG})
+\gamma^{-1}\brk1{\|p_1-p_2\|_{\GG}^2-\|x_1-x_2\|_{\HH}^2}
\nonumber\\
&\geq\inf_{\substack{v\in\GG\\ L^*v=x_1-x_2}}\brk2{
\phi(\|v\|_{\GG})+\gamma^{-1}\brk1{\|v\|_{\GG}-\|L^*v\|^2_{\HH}}}
\nonumber\\
&\geq\inf_{\substack{v\in\GG\\ L^*v=x_1-x_2}}\brk2{
\brk1{\phi\circ\|\cdot\|_{\GG}}^{**}(v)
+\gamma^{-1}\brk1{\|v\|_{\GG}-\|L^*v\|^2_{\HH}}}
\nonumber\\
&=\brk1{L\proxc{\gamma/2}(\phi\circ\|\cdot\|_{\GG})}(x_1-x_2),
\end{align}
which completes the proof.
\end{proof}

\begin{proposition}[Fitzpatrick function]
\label{p:25}
Suppose that $L\in\BL(\HH,\GG)$ satisfies $\|L\|\leq 1$,
let $B\colon\GG\to2^\GG$ be maximally monotone, let
\begin{equation}
\label{e:p25}
F_B\colon\GG\times\GG\to\RXX\colon(x,x^*)\mapsto
\sup_{(v,v^*)\in\gra B}\brk1{\scal{v}{x^*}_{\GG}
+\scal{x}{v^*}_{\GG}-\scal{v}{v^*}_{\GG}}
\end{equation}
be its Fitzpatrick function, and let $\gamma\in\RPP$. Then the
following hold:
\begin{enumerate}
\item
Let $x\in\ker(\Id_{\HH}-L^*\circ L)$ and $x^*\in\HH$. Then
\label{p:25i}
$F_{\gamma\brk1{L\proxc{\gamma}B}}(x,x^*)\leq 
F_{\gamma B}(Lx,Lx^*)$.
\item
\label{p:25ii}
Let $x\in\HH$ and $x^*\in\ker(\Id_{\HH}-L^*\circ L)$. Then
$F_{\gamma\brk1{L\proxcc{\gamma}B}}(x,x^*)\leq 
F_{\gamma B}(Lx,Lx^*)$.
\end{enumerate}
\end{proposition}
\begin{proof}
We recall from Corollary~\ref{c:15}\ref{c:15ii} that
$L\proxc{\gamma}B$ and $L\proxcc{\gamma}B$ are maximally monotone.

\ref{p:25i}: By Minty's parametrization
\cite[Remark~23.23(ii)]{Livre1},
\begin{equation}
\label{e:p25a}
F_{\gamma\brk1{L\proxc{\gamma}B}}(x,x^*)
=\sup_{y\in\HH}\brk3{
\scal2{J_{\gamma\brk1{L\proxc{\gamma}B}}y}{x^*}_{\HH}
+\scal2{x}{y-J_{\gamma\brk1{L\proxc{\gamma}B}}y}_{\HH}
-\scal2{J_{\gamma\brk1{L\proxc{\gamma}B}}y}{y
-J_{\gamma\brk1{L\proxc{\gamma}B}}y}_{\HH}}.
\end{equation}
Thus, by virtue of Proposition~\ref{p:2}\ref{p:2i} and 
$\|L\|\leq 1$, we deduce that, for every $y\in\HH$,
\begin{align}
\label{e:p25b}
&\scal2{x}{y-J_{\gamma\brk1{L\proxc{\gamma}B}}y}_{\HH}
+\scal2{J_{\gamma\brk1{L\proxc{\gamma}B}}y}{x^*}_{\HH}
-\scal2{J_{\gamma\brk1{L\proxc{\gamma}B}}y}{y
-J_{\gamma\brk1{L\proxc{\gamma}B}}y}_{\HH}
\nonumber\\
&=\scal2{x}{y-L^*\brk1{J_{\gamma B}(Ly)}}_{\HH}
+\scal2{L^*\brk1{J_{\gamma B}(Ly)}}{x^*}_{\HH}
-\scal2{L^*\brk1{J_{\gamma B}(Ly)}}{y
-L^*\brk1{J_{\gamma B}(Ly)}}_{\HH}
\nonumber\\
&\leq\scal{x-L^*(Lx)}{y}_{\HH}
+\brk2{\scal{Lx}{Ly-J_{\gamma B}(Ly)}_{\GG}
+\scal{J_{\gamma B}(Ly)}{Lx^*}_{\GG}
-\scal{J_{\gamma B}(Ly)}{Ly-J_{\gamma B}(Ly)}_{\GG}}
\nonumber\\
&\leq\sup_{v\in\GG}\brk2{\scal{Lx}{v-J_{\gamma B}v}_{\GG}
+\scal{J_{\gamma B}v}{Lx^*}_{\GG}
-\scal{J_{\gamma B}v}{v-J_{\gamma B}v}_{\GG}}
\nonumber\\
&=F_{\gamma B}(Lx,Lx^*).
\end{align}
Therefore, taking the supremum over $y\in\HH$ in \eqref{e:p25b},
the conclusion follows from \eqref{e:p25a}.

\ref{p:25ii}: By Proposition~\ref{p:1}\ref{p:1vi}, 
Definition~\ref{d:1}, \eqref{e:p25}, and \ref{p:25i},
\begin{align}
F_{\gamma\brk1{L\proxcc{\gamma}B}}(x,x^*)
=F_{\brk1{L\proxc{}(\gamma B)^{-1}}^{-1}}(x,x^*)
=F_{L\proxc{}(\gamma B)^{-1}}(x^*,x)
\leq F_{(\gamma B)^{-1}}(Lx^*,Lx)
=F_{\gamma B}(Lx,Lx^*),
\end{align}
which completes the proof.
\end{proof}

\begin{corollary}
\label{c:fitz}
Suppose that $L\in\BL(\HH,\GG)$ is an isometry, let
$B\colon\GG\to2^\GG$ be maximally monotone, and let
$\gamma\in\RPP$. Then
\begin{equation}
(\forall x\in\HH)(\forall x^*\in\HH)\quad
F_{\gamma\brk1{L\proxc{\gamma}B}}(x,x^*)\leq
F_{\gamma B}(Lx,Lx^*)
\end{equation}
\end{corollary}
\begin{proof}
Since $L$ is an isometry, $\ker(\Id_{\GG}-L^*\circ L)=\HH$.
Therefore, the conclusion is a consequence of
Proposition~\ref{p:25}\ref{p:25i}.
\end{proof}

\begin{corollary}[\protect{\cite[Theorem~2.13]{Baus16}}]
Let $0\neq p\in\NN$ and let $\gamma\in\RPP$. For every
$k\in\{1,\ldots,p\}$, let $B_k\colon\HH\to2^\HH$ be maximally
monotone and let $\alpha_k\in\RPP$. Suppose that
$\sum_{k=1}^p\alpha_k=1$. Then
\begin{equation}
F_{\gamma\rav_{\gamma}(B_k,\alpha_k)_{1\leq k\leq p}}\leq
\sum_{k=1}^p\alpha_kF_{\gamma B_k}.
\end{equation}
\end{corollary}
\begin{proof}
Define $L$ and $B$ as in Example~\ref{ex:rave} and recall that
$L\proxc{\gamma}B=\rav_{\gamma}(B_k,\alpha_k)_{1\leq k\leq p}$. In
this case, $L$ is an isometry, and it follows from
Corollary~\ref{c:fitz} that, for every $x\in\HH$ and $x^*\in\HH$,
\begin{equation}
F_{\gamma\rav_{\gamma}(B_k,\alpha_k)_{1\leq k\leq p}}(x,x^*)
\leq F_{\gamma B}(Lx,Lx^*)
=\sum_{k=1}^p\alpha_kF_{\gamma B_k}(x,x^*),
\end{equation}
as announced.
\end{proof}

\section{Asymptotic behavior of resolvent compositions}
\label{sec:5}

We examine the convergence of the operators $L\proxc{\gamma}B$ and
$L\proxcc{\gamma}B$ when $\gamma$ varies, studying their
corresponding graph. We begin by recalling some definitions related
to set-convergence, which enable us to characterize the convergence
of operators through their graphs.

\subsection{Set-convergence}

\begin{definition}[Painlev\'{e}--Kuratowski]
Let $(C_n)_{n\in\NN}$ be a sequence of subsets of $\HH$. The
\emph{lower limit} of the sequence $(C_n)_{n\in\NN}$ is the closed
subset of $\HH$ defined by 
\begin{equation}
\varliminf C_n=\menge{x\in\HH}{(\exi
(x_n)_{n\in\NN}\;\;\text{in}\;\;\HH)
(\forall n\in\NN)\,\,x_n\in C_n\,\,\text{and}\,\,x_n\to x}.
\end{equation}
The \emph{upper limit} of the sequence $(C_n)_{n\in\NN}$ is the
closed subset of $\HH$ defined by 
\begin{equation}
\varlimsup C_n=\menge{x\in\HH}{
(\exi (x_n)_{n\in\NN}\;\text{in}\;\HH)
(\exi(k_n)_{n\in\NN}\;\text{in}\;\NN)
(\forall n\in\NN)\,\,x_n\in C_{k_n}\,\,
\text{and}\,\,x_n\to x}.
\end{equation}
The sequence $(C_n)_{n\in\NN}$ is \emph{Painlev\'{e}--Kuratowski}
convergent if its upper limit coincides with its lower limit. The
limit set in this case is given by
\begin{equation}
\lim_{n\to\pinf}C_n=\varliminf C_n=\varlimsup C_n.
\end{equation}
\end{definition}

Let $C$ and $D$ be subsets of $\HH$. The \emph{excess function} of
$C$ on $D$ is defined by
\begin{equation}
e(C,D)=\sup_{x\in C}d_D(x),
\end{equation}
with the convention that $e(\emp,D)=0$.

\begin{definition}[$\rho$-Hausdorff distance \cite{Atto91,Peno90}]
Let $C$ and $D$ be subsets of $\HH$, let $\rho\in\RP$, and set
$C_\rho=C\cap B(0;\rho)$ and $D_\rho=D\cap B(0;\rho)$ The
\emph{$\rho$-Hausdorff distance} between $C$ and $D$ is 
\begin{equation}
\haus_{\rho}(C,D)=\max\{e(C_\rho,D),e(D_\rho,C)\}.
\end{equation}
A sequence $(C_n)_{n\in\NN}$ of subsets of $\HH$ converges with
respect to the $\rho$-Hausdorff distance to the subset $C$ of
$\HH$ if
\begin{equation}
(\forall\rho\in\RPP)\quad\lim_{n\to\pinf}\haus_{\rho}(C_n,C)=0.
\end{equation}
\end{definition}

\subsection{Graph-convergence of operators}

\begin{definition}
Let $(A_n)_{n\in\NN}$ and $A$ be set-valued operators from $\HH$ to
$2^\HH$. Then $(A_n)_{n\in\NN}$ \emph{graph-converges} to
$A$, denoted by $A_n\xrightarrow{g}A$, if $(\gra A_n)_{n\in\NN}$
converges to $\gra A$ in the Painlev\'{e}--Kuratowski sense.
\end{definition}

\begin{definition}
Let $A\colon\HH\to2^\HH$, $B\colon\HH\to2^\HH$, and $\rho\in\RP$.
The \emph{$\rho$-Hausdorff distance} between $A$ and $B$ is
\begin{equation}
\haus_{\rho}(A,B)=\haus_{\rho}(\gra A,\gra B).
\end{equation}
A sequence $(A_n)_{n\in\NN}$ of operators from $\HH$ to $2^\HH$
converges with respect to the $\rho$-Hausdorff distance to the
operator $A\colon\HH\to2^\HH$ if
\begin{equation}
(\forall \rho\in\RPP)\quad
\lim_{n\to\pinf}\haus_{\rho}(A_n,A)=0.
\end{equation}
\end{definition}

Some equivalences of graph-convergence for maximally monotone
operators are summarized in the following result (see e.g.
\cite{Atto84}).

\begin{lemma}
\label{l:57}
Let $(A_n)_{n\in\NN}$ and $A$ be maximally monotone operators from
$\HH$ to $2^\HH$. Then the following are equivalent:
\begin{enumerate}
\item
$A_n\xrightarrow{g}A$.
\item
$(\forall \gamma\in\RPP)(\forall x\in\HH)\,\,
J_{\gamma A_n}x\to J_{\gamma A}x$.
\item
$(\exi \gamma\in\RPP)(\forall x\in\HH)\,\,
J_{\gamma A_n}x\to J_{\gamma A}x$.
\end{enumerate}
\end{lemma}

\begin{lemma}
\label{l:69}
Let $(A_n)_{n\in\NN}$ and $A$ be maximally monotone operators from
$\HH$ to $2^\HH$, and let $(\gamma_n)_{n\in\NN}$ and $\gamma$ be in
$\RPP$. Suppose that $A_n\xrightarrow{g}A$ and
$\gamma_n\to\gamma$. Then the following hold:
\begin{enumerate}
\item
\label{l:69i}
$A_n^{-1}\xrightarrow{g}A^{-1}$.
\item
\label{l:69ii}
$\gamma_n A_n\xrightarrow{g}\gamma A$
\end{enumerate}
\end{lemma}
\begin{proof}
\ref{l:69i}: This follows from \eqref{e:23} and Lemma~\ref{l:57}.

\ref{l:69ii}: Let $x\in\HH$ and set $(\forall n\in\NN)$
$\theta_n=1-\gamma_n/\gamma$. By
\cite[Proposition~23.31(iii)]{Livre1},
\begin{align}
\label{e:l69}
\|J_{\gamma_n A_n}x-J_{\gamma A}x\|_{\HH}
&\leq\|J_{\gamma_n A_n}x-J_{\gamma A_n}x\|_{\HH}
+\|J_{\gamma A_n}x-J_{\gamma A}x\|_{\HH}
\nonumber\\
&\leq\lvert\theta_n\rvert\|x-J_{\gamma A_n}x\|_{\HH}
+\|J_{\gamma A_n}x-J_{\gamma A}x\|_{\HH}.
\end{align}
Further, $A_n\xrightarrow{g}A$ and Lemma~\ref{l:57} yield
$J_{\gamma A_n}x\to J_{\gamma A}x$. Altogether, since 
$\theta_n\to 0$, we deduce from \eqref{e:l69} that
$J_{\gamma_nA_n}x\to J_{\gamma A}x$. Finally, invoking
Lemma~\ref{l:57} one more, we obtain the assertion.
\end{proof}

\begin{lemma}[\protect{\cite[Propositions~1.1 and 1.2]{Atto93}}]
\label{l:58}
Let $A_1\colon\HH\to2^\HH$ and $A_2\colon\HH\to2^\HH$ be maximally
monotone, and let $\gamma\in\RPP$. Consider 
\begin{equation}
\label{e:hausd}
(\forall\delta\in\RP)\quad
d_{\gamma,\delta}(A_1,A_2)
=\sup_{x\in B(0;\delta)}\|J_{\gamma A_1}x-J_{\gamma A_2}x\|_{\HH}.
\end{equation}
Then the following hold:
\begin{enumerate}
\item
\label{l:58i}
$(\forall\rho\in\RP)$
$\haus_{\rho}(A_1,A_2)\leq
\max\{1,\gamma^{-1}\}d_{\gamma,(1+\gamma)\rho}(A_1,A_2)$.
\item
\label{l:58ii}
Set $\rho=\max\{\delta+\|J_{\gamma A_1}0\|_{\HH},
\gamma^{-1}(\delta+\|J_{\gamma A_1}0\|_{\HH}\}$. Then 
$d_{\gamma,\delta}(A_1,A_2)\leq(2+\gamma)\haus_{\rho}(A_1,A_2)$.
\end{enumerate}
\end{lemma}

\subsection{Convergence of resolvent compositions}

We proceed to study the graph-convergence and convergence with
respect to the $\rho$-Hausdorff distance of resolvent compositions.

\begin{proposition}
\label{p:70}
Let $(L_n)_{n\in\NN}$ and $L$ be in $\BL(\HH,\GG)$, let
$(B_n)_{n\in\NN}$ and $B$ be maximally monotone operators from
$\GG$ to $2^\GG$, and let $(\gamma_n)_{n\in\NN}$ and $\gamma$ be in
$\RPP$. Suppose that $L_n\to L$, $B_n\xrightarrow{g}B$,
$\gamma_n\to\gamma$, and $(\forall n\in\NN)$ $\|L_n\|\leq 1$.
Then the following hold:
\begin{enumerate}
\item
\label{p:70i}
$L_n\proxc{\gamma_n}B_n\xrightarrow{g}L\proxc{\gamma}B$.
\item
\label{p:70ii}
$L_n\proxcc{\gamma_n}B_n\xrightarrow{g}L\proxcc{\gamma}B$.
\end{enumerate}
\end{proposition}
\begin{proof}
We recall from Corollary~\ref{c:15}\ref{c:15ii} that the operators
$(L_n\proxc{\gamma_n}B_n)_{n\in\NN}$, $L\proxc{\gamma}B$,
$(L_n\proxcc{\gamma_n}B_n)_{n\in\NN}$, and $L\proxcc{\gamma}B$ are
maximally monotone. Therefore, by Lemma~\ref{l:57}, it is
enough to verify the convergence of the resolvent of these
operators.

\ref{p:70i}: Let $x\in\HH$ and set $(\forall n\in\NN)$
$\theta_n=1-\gamma/\gamma_n$. It follows from
\cite[Proposition~23.31(iii)]{Livre1},
Proposition~\ref{p:2}\ref{p:2i}, and Lemma~\ref{l:69} that
\begin{align}
\label{e:p70a}
\|J_{\gamma\brk1{L_n\proxc{\gamma_n}B_n}}x
-J_{\gamma\brk1{L\proxc{\gamma}B}}x\|_{\HH}&\leq
\|J_{\gamma\brk1{L_n\proxc{\gamma_n}B_n}}x
-J_{\gamma_n\brk1{L_n\proxc{\gamma_n}B_n}}x\|_{\HH}
+\|J_{\gamma_n\brk1{L_n\proxc{\gamma_n}B_n}}x
-J_{\gamma\brk1{L\proxc{\gamma}B}}x\|_{\HH}\nonumber\\
&\leq\lvert\theta_n\rvert\,
\|x-J_{\gamma_n\brk1{L_n\proxc{\gamma_n}B_n}}x\|_{\HH}
+\|L_n^*\brk1{J_{\gamma B_n}(L_nx)}-
L^*\brk1{J_{\gamma B}(Lx)}\|_{\HH}\nonumber\\
&=\lvert\theta_n\rvert\,
\|x-L_n^*\brk1{J_{\gamma_nB_n}(L_nx)}\|_{\HH}
+\|L_n^*\brk1{J_{\gamma B_n}(L_nx)}-
L^*\brk1{J_{\gamma B}(Lx)}\|_{\HH}.
\end{align}
Further, nonexpansiveness of $J_{\gamma_nB_n}$ implies that
\begin{align}
\label{e:p70b}
\|J_{\gamma_n B_n}(L_nx)-J_{\gamma B}(Lx)\|_{\GG}&\leq
\|J_{\gamma_n B_n}(L_nx)-J_{\gamma_n B_n}(Lx)\|_{\GG}
+\|J_{\gamma_n B_n}(Lx)-J_{\gamma B}(Lx)\|_{\GG}\nonumber\\
&\leq\|L_nx-Lx\|_{\GG}
+\|J_{\gamma_n B_n}(Lx)-J_{\gamma B}(Lx)\|_{\GG}.
\end{align}
On the other hand, by Lemma~\ref{l:69}\ref{l:69ii},
$\gamma_nB\xrightarrow{g}\gamma B$. Since $\theta_n\to 0$ and
$L_n\to L$, we combine \eqref{e:p70a} and \eqref{e:p70b} to obtain 
$J_{\gamma\brk1{L_n\proxc{\gamma_n}B_n}}x\to
J_{\gamma\brk1{L\proxc{\gamma}B}}x$. Therefore, the conclusion
follows from Lemma~\ref{l:57}.

\ref{p:70ii}: By Lemma~\ref{l:69}\ref{l:69i},
$B_n^{-1}\xrightarrow{g}B^{-1}$. Therefore, \ref{p:70i} yields
$L_n\proxc{1/\gamma_n}B_n^{-1}\xrightarrow{g}
L\proxc{1/\gamma}B^{-1}$. Altogether, by Definition~\ref{d:1} and
Lemma~\ref{l:69}\ref{l:69i} once more, 
\begin{equation}
L_n\proxcc{\gamma_n}B_n=\brk1{L_n\proxc{1/\gamma_n}B_n^{-1}}^{-1}
\xrightarrow{g}\brk1{L\proxc{1/\gamma}B^{-1}}^{-1}
=L\proxcc{\gamma}B,
\end{equation}
as asserted.
\end{proof}

\begin{proposition}
\label{p:510}
Let $(L_n)_{n\in\NN}$ and $L$ be in $\BL(\HH,\GG)$, let
$B\colon\GG\to2^\GG$ be maximally monotone, and let
$(\gamma_n)_{n\in\NN}$ be a sequence in $\RPP$. Suppose that
$L_n\to L$ and that $(\forall n\in\NN)$ $\|L_n\|\leq 1$. Then the
following hold:
\begin{enumerate}
\item
\label{p:510i}
Suppose that $\gamma_n\downarrow 0$. Then the following are
satisfied:
\begin{enumerate}
\item
\label{p:510ia}
$L_n\proxc{}(\gamma_n B)\xrightarrow{g}L\proxc{}N_{\cdom B}$.
\item
\label{p:510ib}
$L_n\proxcc{}(\gamma_n B)\xrightarrow{g}L\proxcc{}N_{\cdom B}$.
\end{enumerate}
\item
\label{p:510ii}
Suppose that $\gamma_n\uparrow\pinf$ and that $\zer B\neq\emp$.
Then the following are satisfied:
\begin{enumerate}
\item
\label{p:510iia}
$L_n\proxc{}(\gamma_n B)\xrightarrow{g}L\proxc{}N_{\zer B}$.
\item
\label{p:510iib}
$L_n\proxcc{}(\gamma_n B)\xrightarrow{g}L\proxcc{}N_{\zer B}$.
\end{enumerate}
\end{enumerate}
\end{proposition}
\begin{proof}
\ref{p:510i}:
Let $y\in\GG$. We recall from \cite[Corollary~21.14]{Livre1} that
$\cdom B$ is closed and convex. Further, by 
\cite[Theorem~23.48]{Livre1},
$J_{\gamma_n B}y\to\proj_{\cdom B}y=J_{N_{\cdom B}}y$. Therefore,
it follows from Lemma~\ref{l:57} that
$\gamma_nB\xrightarrow{g}N_{\cdom B}$, and the conclusion is a
consequence of Proposition~\ref{p:70}.

\ref{p:510ii}: 
Let $y\in\GG$. We recall from \cite[Proposition~23.39]{Livre1} that
$\zer B$ is closed and convex. Further, by
\cite[Theorem~23.48]{Livre1}, 
$J_{\gamma_n B}y\to\proj_{\zer B}y=J_{N_{\zer B}}y$. Therefore,
it follows from Lemma~\ref{l:57} that
$\gamma_nB\xrightarrow{g}N_{\zer B}$, and the conclusion is a
consequence of Proposition~\ref{p:70}.
\end{proof}

The following proposition shows that, for a fixed $\gamma\in\RPP$,
resolvent compositions are nonexpansive with respect to
$d_{\gamma,\delta}$.

\begin{proposition}
\label{p:73}
Suppose that $L\in\BL(\HH,\GG)$ satisfies $\|L\|\leq 1$, let
$B_1\colon\GG\to2^\GG$ and $B_2\colon\GG\to2^\GG$ be maximally
monotone, let $\gamma\in\RPP$, and let $\delta\in\RPP$. Then the
following hold:
\begin{enumerate}
\item
\label{p:73i}
$d_{\gamma,\delta}(L\proxc{\gamma}B_1,L\proxc{\gamma}B_2)
=d_{\gamma,\delta}(L\proxcc{\gamma}B_1,L\proxcc{\gamma}B_2)$.
\item
\label{p:73ii}
$d_{\gamma,\delta}(L\proxc{\gamma}B_1,L\proxc{\gamma}B_2)
\leq\|L\|\,d_{\gamma,\|L\|\delta}(B_1,B_2)$.
\item
\label{p:73iii}
$d_{\gamma,\delta}(L\proxcc{\gamma}B_1,L\proxcc{\gamma}B_2)
\leq\|L\|\,d_{\gamma,\|L\|\delta}(B_1,B_2)$.
\end{enumerate}
\end{proposition}
\begin{proof}
We recall from Corollary~\ref{c:15}\ref{c:15ii} that, for every
$k\in\{1,2\}$, $L\proxc{\gamma}B_k$ and
$L\proxcc{\gamma}B_k$ are maximally monotone.

\ref{p:73i}: By Proposition~\ref{p:2},
$J_{\gamma\brk1{L\proxc{\gamma}B_1}}-
J_{\gamma\brk1{L\proxc{\gamma}B_2}}
=J_{\gamma\brk1{L\proxcc{\gamma}B_1}}-
J_{\gamma\brk1{L\proxcc{\gamma}B_2}}$. Therefore, the conclusion
follows from \eqref{e:hausd}.

\ref{p:73ii}: Let $x\in B(0;\delta)$. It follows from
Proposition~\ref{p:2}\ref{p:2i} that
\begin{align}
\label{e:p73}
\|J_{\gamma\brk1{L\proxc{\gamma}B_1}}x-
J_{\gamma\brk1{L\proxc{\gamma}B_2}}x\|_{\HH}
&=\|L^*\brk1{J_{\gamma B_1}(Lx)}
-L^*\brk1{J_{\gamma B_2}(Lx)}\|_{\HH}
\nonumber\\
&\leq\|L\|\,\|J_{\gamma B_1}(Lx)-J_{\gamma B_2}(Lx)\|_{\GG}
\nonumber\\
&\leq\|L\|\sup_{u\in B(0;\|L\|\delta)}
\|J_{\gamma B_1}u-J_{\gamma B_2}u\|_{\GG}\nonumber\\
&=\|L\|\,d_{\gamma,\|L\|\delta}(B_1,B_2).
\end{align}
Therefore, by taking the supremum over all $x\in B(0;\delta)$, we
obtain the assertion.

\ref{p:73iii}: A consequence of \ref{p:73i} and \ref{p:73ii}.
\end{proof}

\begin{proposition}
Suppose that $L\in\BL(\HH,\GG)$ satisfies $\|L\|\leq 1$, let
$B_1\colon\GG\to2^\GG$ and $B_2\colon\GG\to2^\GG$ be maximally
monotone, let $\gamma\in\RPP$, and let $\rho\in\RPP$. Then
\begin{equation}
\haus_{\rho}(L\proxc{\gamma}B_1,L\proxc{\gamma}B_2)\leq
\max\{1,\gamma^{-1}\}\|L\|d_{\gamma,\|L\|(1+\gamma)\rho}(B_1,B_2).
\end{equation}
\end{proposition}
\begin{proof}
Combine Corollary~\ref{c:15}\ref{c:15ii},
Lemma~\ref{l:58}\ref{l:58i}, and
Proposition~\ref{p:73}\ref{p:73ii}.
\end{proof}

\begin{proposition}
\label{p:74}
Suppose that $L\in\BL(\HH,\GG)$ satisfies $\|L\|\leq 1$ and let
$B\colon\GG\to2^\GG$ be maximally monotone. Assume that 
$L^*\circ B\circ L$ is maximally monotone. Then the following hold:
\begin{enumerate}
\item
\label{p:74i}
$L\proxcc{\gamma}B\xrightarrow{g}L^*\circ B\circ L$ as 
$0<\gamma\to 0$.
\item
\label{p:74ii}
Suppose that one of the following is satisfied:
\begin{enumerate}
\item
\label{p:74iia}
$\ran(B\circ L)$ is bounded.
\item
\label{p:74iib}
There exists $S\in\BL(\HH,\GG)$ such that $S\circ L^*$ is
invertible.
\end{enumerate}
Then
\begin{equation}
(\forall\rho\in\RPP)\quad\lim_{\gamma\to 0}
\haus_{\rho}(L\proxcc{\gamma}B,L^*\circ B\circ L)=0.
\end{equation}
\end{enumerate}
\end{proposition}
\begin{proof}
Let $\gamma\in\zeroun$ and recall from
Corollary~\ref{c:15}\ref{c:15ii} that $L\proxcc{\gamma}B$ is
maximally monotone. Let $x\in\HH$, let $\rho\in\RPP$, and suppose
that $x\in B(0;2\rho)$. Set $\Psi=\Id_{\GG}-L\circ L^*$, 
set $p=J_{L^*\circ B\circ L}x$, and set
$p_\gamma=J_{L\proxcc{\gamma}B}x$. We deduce from
Proposition~\ref{p:3}\ref{p:3ii} that
\begin{align}
\label{e:p71a}
x-p_\gamma\in\brk1{L\proxcc{\gamma}B}p_\gamma&\Leftrightarrow
x-p_\gamma\in L^*\brk2{\brk1{B^{-1}+\gamma\Psi}^{-1}(Lp_\gamma)}
\nonumber\\
&\Leftrightarrow
\begin{cases}
(\exi y_\gamma\in\GG)\quad x-p_\gamma=L^*y_\gamma\\
Lp_\gamma\in\brk1{B^{-1}+\gamma\Psi}y_\gamma
\end{cases}\nonumber\\
&\Leftrightarrow
\begin{cases}
(\exi y_\gamma\in\GG)\quad x-p_\gamma=L^*y_\gamma\\
y_\gamma\in B(Lp_\gamma-\gamma\Psi y_\gamma).
\end{cases}
\end{align}
On the other hand,
\begin{equation}
\label{e:p71b}
x-p\in L^*\brk1{B(Lp)}\Leftrightarrow
(\exi y\in\GG)\quad x-p=L^*y\quad\text{and}\quad y\in B(Lp).
\end{equation}
Altogether, monotonicity of $B$, \eqref{e:p71a}, and 
\eqref{e:p71b} yield
\begin{align}
\label{e:p71c}
\scal{(Lp_\gamma-\gamma\Psi y_\gamma)-Lp}{y_\gamma-y}_{\GG}\geq 0
&\Leftrightarrow
\scal{p_\gamma-p}{L^*(y_\gamma-y)}_{\HH}
-\gamma\scal{\Psi y_\gamma}{y_\gamma-y}_{\GG}\geq 0\nonumber\\
&\Leftrightarrow\scal{p_\gamma-p}{p-p_\gamma}_{\HH}
-\gamma\scal{\Psi y_\gamma}{y_\gamma-y}_{\GG}\geq 0\nonumber\\
&\Leftrightarrow\|p_\gamma-p\|_{\HH}^2
+\gamma\scal{\Psi y_\gamma}{y_\gamma-y}_{\GG}\leq 0.
\end{align}
Further, since $L^*y_\gamma=x-p_\gamma$ and $L^*y=x-p$, by
Cauchy--Schwarz inequality \cite[Fact~2.11]{Livre1},
\begin{align}
\label{e:p71d}
\scal{\Psi y_\gamma}{y_\gamma-y}_{\GG}
&=\scal{y_\gamma-L(x-p_\gamma)}{y_\gamma-y}_{\GG}
\nonumber\\
&=\scal{y_\gamma}{y_\gamma-y}_{\GG}
-\scal{x-p_\gamma}{L^*(y_\gamma-y)}_{\HH}
\nonumber\\
&=\|y_\gamma\|_{\GG}^2-\scal{y_\gamma}{y}_{\GG}
-\scal{x-p_\gamma}{p-p_\gamma}_{\HH}
\nonumber\\
&\geq\|y_\gamma\|_{\GG}^2-\|y_\gamma\|_{\GG}\,\|y\|_{\GG}
-\brk1{\scal{p-p_\gamma}{p-p_\gamma}_{\HH}
+\scal{x-p}{p-p_\gamma}_{\HH}}
\nonumber\\
&\geq\|y_\gamma\|_{\GG}^2-\|y_\gamma\|_{\GG}\,\|y\|_{\GG}
-\|p-p_\gamma\|_{\HH}^2-\|x-p\|_{\HH}\,\|p-p_\gamma\|_{\HH}
\nonumber\\
&\geq\min_{\alpha\in\RR}\brk1{\alpha^2-\alpha\|y\|_{\GG}}
-\|p-p_\gamma\|_{\HH}^2-\|x-p\|_{\HH}\,\|p-p_\gamma\|_{\HH}
\nonumber\\
&=-\frac{1}{4}\|y\|_{\GG}^2-\|p-p_\gamma\|_{\HH}^2
-\|x-p\|_{\HH}\,\|p-p_\gamma\|_{\HH}.
\end{align}
Set $\delta=2\rho+\|J_{L^*\circ B\circ L}0\|_{\HH}$. Since 
$L^*\circ B\circ L$ is maximally monotone, nonexpansiveness
of $J_{L^*\circ B\circ L}$ yields
\begin{equation}
\label{e:p71e}
\|p\|_{\HH}\leq\|J_{L^*\circ B\circ L}x
-J_{L^*\circ B\circ L}0\|_{\HH}
+\|J_{L^*\circ B\circ L}0\|_{\HH}
\leq\|x\|_{\HH}+\|J_{L^*\circ B\circ L}0\|_{\HH}
\leq 2\rho+\|J_{L^*\circ B\circ L}0\|_{\HH}=\delta.
\end{equation}
Thus, we combine \eqref{e:p71c}, \eqref{e:p71d}, and \eqref{e:p71e}
to deduce that
\begin{align}
\label{e:p71f}
&\|p-p_\gamma\|_{\HH}^2-\frac{\gamma}{4}\|y\|_{\GG}^2
-\gamma\|p-p_\gamma\|_{\HH}^2
-\gamma\|x-p\|_{\HH}\,\|p-p_\gamma\|_{\HH}\leq 0
\nonumber\\
&\Leftrightarrow
(1-\gamma)\|p-p_\gamma\|_{\HH}^2
-\gamma\|x-p\|_{\HH}\,\|p-p_\gamma\|_{\HH}
-\frac{\gamma}{4}\|y\|_{\GG}^2\leq 0\nonumber\\
&\Leftrightarrow
\|p-p_\gamma\|_{\HH}^2-\frac{\gamma}{1-\gamma}
\|x-p\|_{\HH}\,\|p-p_\gamma\|_{\HH}
-\frac{\gamma}{4(1-\gamma)}\|y\|_{\GG}^2\leq 0\nonumber\\
&\Rightarrow
\|p-p_\gamma\|_{\HH}^2-\frac{\gamma}{1-\gamma}
(\|x\|_{\HH}+\|p\|_{\HH})\|p-p_\gamma\|_{\HH}
-\frac{\gamma}{4(1-\gamma)}\|y\|_{\GG}^2\leq 0\nonumber\\
&\Rightarrow
\|p-p_\gamma\|_{\HH}^2
-\frac{\gamma}{1-\gamma}(2\rho+\delta)\|p-p_\gamma\|_{\HH}
-\frac{\gamma}{4(1-\gamma)}\|y\|_{\GG}^2\leq 0\nonumber\\
&\Rightarrow
\brk3{\|p-p_\gamma\|_{\HH}-\frac{\gamma}{2(1-\gamma)}
(2\rho+\delta)}^2
\leq\frac{\gamma^2}{4(1-\gamma)^2}(2\rho+\delta)^2
+\frac{\gamma}{4(1-\gamma)}\|y\|_{\GG}^2\nonumber\\
&\Rightarrow
\|p-p_\gamma\|_{\HH}
\leq\brk3{\frac{\gamma^2}{4(1-\gamma)^2}(2\rho+\delta)^2
+\frac{\gamma}{4(1-\gamma)}\|y\|_{\GG}^2}^{1/2}
+\frac{\gamma}{2(1-\gamma)}(2\rho+\delta).
\end{align}

\ref{p:74i}: By \eqref{e:p71f},
$J_{L\proxcc{\gamma}B}x\to J_{L^*\circ B\circ L}x$ as
$0<\gamma\to 0$, and the conclusion follows from Lemma~\ref{l:57}.

\ref{p:74ii}: Assumption~\ref{p:74iia} implies that there
exits $\eta\in\RPP$ such that, for every $z\in\ran(B\circ L)$,
$\|z\|_{\GG}\leq\eta$. In particular, \eqref{e:p71b} yields
$\|y\|_{\GG}\leq\eta$. On the other hand,
Assumption~\ref{p:74iib} and \eqref{e:p71b} imply that
$y=(S\circ L^*)^{-1}(S(x-p))$. Thus,
$\|y\|_{\GG}\leq\|(S\circ L^*)^{-1}\|\,\|S\|(2\rho+\delta)$.
Therefore, either Assumption~\ref{p:74iia} or
Assumption~\ref{p:74iib} implies that there exists $\eta\in\RPP$
such that $\|y\|_{\GG}\leq\eta$. Altogether, we deduce from
Lemma~\ref{l:58}\ref{l:58i} and \eqref{e:p71f} that
\begin{align}
\haus_{\rho}(L\proxcc{\gamma}B,L^*\circ B\circ L)&\leq
d_{1,2\rho}(L\proxcc{\gamma}B,L^*\circ B\circ L)\nonumber\\
&=\sup_{u\in B(0;2\rho)}
\|J_{L\proxcc{\gamma}B}u-J_{L^*\circ B\circ L}u\|_{\HH}
\nonumber\\
&\leq\brk3{\frac{\gamma^2}{4(1-\gamma)^2}(2\rho+\delta)^2
+\frac{\gamma}{4(1-\gamma)}\eta^2}^{1/2}
+\frac{\gamma}{2(1-\gamma)}(2\rho+\delta)
\nonumber\\
&\to 0\;\;\text{as}\;\;0<\gamma\to 0,
\end{align}
which completes the proof.
\end{proof}

\begin{corollary}
\label{c:72}
Let $0\neq p\in\NN$ and, for every $k\in\{1,\ldots,p\}$,
let $\GG_k$ be a real Hilbert space, let $L_k\in\BL(\HH,\GG_k)$, 
let $B_k\colon\GG_k\to2^{\GG_k}$ be maximally monotone, and let
$\alpha_k\in\RPP$. Suppose that
$\sum_{k=1}^p\alpha_k\|L_k\|^2\leq 1$ and that 
$\sum_{k=1}^p\alpha_kL_k^*\circ B_k\circ L_k$ is maximally
monotone. Then
\begin{equation}
\Rcm{\gamma}(B_k,L_k)_{1\leq k\leq p}\xrightarrow{g}
\sum_{k=1}^p\alpha_kL_k^*\circ B_k\circ L_k
\;\;\text{as}\;\;0<\gamma\to 0.
\end{equation}
\end{corollary}
\begin{proof}
Define $L$ as in \eqref{e:mix3} and $B$ as in \eqref{e:mix4}, and
recall that from Example~\ref{ex:mix2} that, for every
$\gamma\in\RPP$,
$\Rcm{\gamma}(B_k,\alpha_k)_{1\leq k\leq p}=L\proxcc{\gamma}B$ is
maximally monotone. Therefore, since
$\sum_{k=1}^p\alpha_kL_k^*\circ B_k\circ L_k=L^*\circ B\circ L$,
the conclusion follows from Proposition~\ref{p:74}\ref{p:74i}.
\end{proof}

\begin{corollary}[\protect{\cite[Proposition~1.4]{Atto93}}]
\label{c:515}
Let $A\colon\HH\to2^\HH$ be maximally monotone. Then
\begin{equation}
\label{e:c515}
(\forall\rho\in\RPP)\quad
\lim_{\gamma\to 0}\haus_{\rho}(\moyo{A}{\gamma},A)=0.
\end{equation}
\end{corollary}
\begin{proof}
Set $L=\Id_{\HH}/2$ and $B=2A(2\Id_{\HH})$. Thus, 
$L^*\circ B\circ L=A$. Further, by Example~\ref{ex:yosi},
$(\forall\gamma\in\RPP)$ $\moyo{A}{\gamma}=L\proxcc{\gamma/3}B$.
Since $\Id_{\HH}\circ L^*$ is invertible, we derive from
Proposition~\ref{p:74}\ref{p:74iib} that
\begin{align}
(\forall\rho\in\RPP)\quad
\haus_{\rho}(\moyo{A}{\gamma},A)
=\haus_{\rho}(L\proxcc{\gamma/3}B,L^*\circ B\circ L)
\to 0\;\;\text{as}\;\;0<\gamma\to 0,
\end{align}
which establishes \eqref{e:c515}.
\end{proof}

\begin{corollary}
\label{c:516}
Suppose that $L\in\BL(\HH,\GG)$ satisfies $0<\|L\|\leq 1$ and let
$g\in\Gamma_0(\GG)$. Assume that $0\in\sri(\dom g-\ran L)$. Then
the following hold:
\begin{enumerate}
\item
\label{c:516i}
$\partial(L\proxcc{\gamma}g)\xrightarrow{g}\partial(g\circ L)$
as $0<\gamma\to 0$.
\item
\label{c:516ii}
Suppose that $g\colon\GG\to\RR$ is $\beta$-Lipschitzian for some
$\beta\in\RPP$. Then
\begin{equation}
(\forall\rho\in\RPP)\quad
\lim_{\gamma\to 0}\haus_{\rho}
\brk1{\partial(L\proxcc{\gamma}g),\partial(g\circ L)}=0.
\end{equation}
\end{enumerate}
\end{corollary}
\begin{proof}
Invoking \cite[Corollaries~13.38 and 16.53(i)]{Livre1}, $g^{**}=g$
and $\partial(g\circ L)=L^*\circ(\partial g)\circ L$. 

\ref{c:516i}:
It follows from Example~\ref{ex:proxc}\ref{ex:proxcii} and
Proposition~\ref{p:74}\ref{p:74i} that
\begin{equation}
\partial(L\proxcc{\gamma}g)
=L\proxcc{\gamma}\partial g\xrightarrow{g}
L^*\circ(\partial g)\circ L=\partial(g\circ L)
\;\;\text{as}\;\;0<\gamma\to 0.
\end{equation}

\ref{c:516ii}: Appealing to \cite[Corollary~17.19]{Livre1}, 
$\ran\partial g\subset B(0;\beta)$. Thus, $\ran\partial g$ is
bounded, and the conclusion follows from
Proposition~\ref{p:74}\ref{p:74iia}.
\end{proof}

\section*{Acknowledgement}

This work is a part of the author's Ph.D. dissertation. The author
gratefully acknowledges the guidance of his Ph.D. advisor P. L.
Combettes, throughout this work.

\end{document}